\documentclass[a4paper, 12pt]{amsart}
\usepackage{amssymb,amsmath,amsthm,bm,color}
\allowdisplaybreaks[1]
\theoremstyle{definition}
\newtheorem{definition}{Definition}[section]
\theoremstyle{plain}
\newtheorem{theorem}{Theorem}[section]
\newtheorem{lemma}{Lemma}[section]
\newtheorem{proposition}{Proposition}[section]

\newtheorem{conjecture}{Conjecture}
\theoremstyle{remark}
\newtheorem{remark}{Remark}[section]
\makeatletter

\@addtoreset{equation}{section}
\makeatother
\newcommand{\Boxed}[1]{\mathbin{\ooalign{$\Box$\crcr
  \hidewidth
  \raise0.55ex\hbox{$\scriptscriptstyle{#1}$}%
  \hidewidth
}}}
\newcommand{\g}{\mathfrak{g}}
\newcommand{\W}{\mathcal{W}}
\newcommand{\Z}{\mathbb{Z}}
\newcommand{\Q}{\mathbb{Q}}
\newcommand{\C}{\mathbb{C}}
\newcommand{\cA}{\mathcal{A}}

\title[Fourth order MLDEs and W-algebras]{Modular linear differential equations of fourth order and minimal $\mathcal{W}$-algebras}
\author{Kazuya Kawasetsu\and Yuichi Sakai}
\address{School of Mathematics and Statistics, the University of Melbourne, 3010, Australia  \and Multiple Zeta Research Center, Kyushu University, 744, Motooka, Nishi-ku, Fukuoka 819-0395, Japan.}
\email{kazuya.kawasetsu@unimelb.edu.au \and dynamixaxs@gmail.com}

\date{}

\renewcommand{\=}{\,=\,}

\begin{document}

\begin{abstract}
A~characterization of the minimal $\W$-algebras associated with
the Deligne exceptional series at level $-h^\vee/6$ is obtained by using
one-parameter family  of modular linear differential equations of order $4$.
In particular, the characters of the Ramond-twisted modules of minimal $\mathcal{W}$-algebras related to the Deligne exceptional series satisfy one of these differential equations.
In order to obtain the characterization, the differential equations in the one parameter family which have solutions of ``CFT type" are
classified, whose solutions are explicitly described.
\end{abstract}

\maketitle

\section{Introduction}
The modular linear differential equations (MLDEs) play an~important role in the study of
 2D~conformal field theories, the theory of vertex operator algebras (VOAs) and 
the theory of modular forms (see e.g.~\cite{MMS}, \cite{KZ2}, 
\cite{DLM}, \cite{Mil}, \cite{T} and \cite{AKNS}).
For example, they have appeared in attempts to classify $C_2$-cofinite, rational VOAs and rational conformal field theories (RCFTs)
from their characters.

S.~Mathur, S.~Mukhi and A.~Sen classified the unitary RCFTs
 whose characters  satisfy the second order MLDE of weight 0
\begin{equation}\label{eqn:second}
f''-\frac{1}{6}E_2f'-s E_4f=0\qquad (s\in\mathbb{C}),
\tag{$\sharp_s$}
\end{equation}
where $f$ is a~function on the complex upper-half plane $\mathbb{H}$ and $E_k$ is the normalized Eisenstein series
of weight $k$ (\cite{MMS} and \cite{MMS2}).
Here
\[
f'=D(f)=\frac{1}{2\pi\sqrt{-1}}\frac{df}{d\tau}=q\frac {df}{dq},\qquad (q=e^{2\pi i\tau},
\ \tau\in\mathbb{H}).
\]
(The eq.~$(\sharp_{\mu(t)})$, where $\mu(t)=t(t+2)/144$ for $t\in\mathbb{Q}$, 
is  known as the {\em Kaneko-Zagier equation}~\cite{KZ2}.)

In~\cite{MMS} and~\cite{MMS2}, they showed that such RCFTs coincide with the Wess-Zumino-Witten models (WZW models) at level $1$ associated to the 
Lie algebras in the sequence
$A_1\subset A_2 \subset G_2 \subset D_4 \subset F_4 \subset E_6 \subset E_7 \subset E_8
$, which is called the {\em Deligne exceptional series} (\cite{D}).
Note that the chiral algebra of the WZW model at level $k\in\Z_{\geq 0}$
associated to a~simple Lie algebra $\g$
coincides with the {\it affine VOA} at level $k$ associated with the same Lie algebra $\g$.

The main theme of this paper is the following relations
between MLDEs and VOAs associated with the Deligne
exceptional series.
We first prove a~vertex operator algebraic analogue of the above result
in~\cite{MMS2}.
Let $V$ be a~self-dual simple, $C_2$-cofinite, rational VOA of CFT type 
such that the conformal weights of the $V$-modules are all non-negative 
and that the characters of modules over $V$
satisfy~$(\sharp_s)$ for some $s\in \C$.
We show that $V$ is isomorphic to the affine VOA at level~$1$
associated with a Lie algebra in the Deligne exceptional series (Theorem~\ref{sec:propsecond} and Remark~\ref{sec:mms} (ii)).

Moreover, we find other MLDEs related to the Deligne exceptional series:
 we establish a~characterization of {\it $\mathcal{W}$-algebras} associated with the Deligne exceptional series in terms of the following one-parameter family of MLDEs of order $4$.
Let $s$ be a~complex number.
We study the MLDE of order $4$ and weight 0 which has the form
\begin{align}\label{eqn:fourth}
&D^4(f)-E_2 D^3(f)+(3D(E_2)+\alpha_1(s) E_4)D^2(f)\nonumber\\
&\quad -\Bigl(D^2(E_2)+\frac{\alpha_1(s)}{2}D(E_4)-\alpha_2(s) 
E_6\Bigr)D(f)+\alpha_3(s) E_8 f=0,\tag{$\flat_s$}
\end{align}
where 
\[
\alpha_1=\frac{-25s^2+120s+1332}{7200}, \quad
\alpha_2=\frac{(5s+6)^2}{14400},\quad
\alpha_3=\frac{(s-18)(s+6)(5s+6)^2}{8294400}.
\]
In \S~\ref{sec:cft},
we give a~necessary condition of $s\in \mathbb{C}$ such that~\eqref{eqn:fourth}
has a~solution of CFT type, where a~function $f$ is called
of CFT type if
$f$ has a~$q$-expansion 
$
f(\tau)=q^\alpha(1+\sum_{n=1}^\infty a_nq^n)
$
, where $\alpha$ is a~rational number and $a_n$ is a~non-negative integer for each positive integer $n$.
Moreover, we construct exact solutions of~\eqref{eqn:fourth} for these values of $s$ in Appendices~\ref{sec:basisrational}--\ref{sec:quasimodular}.

The space spanned by the characters of modules over a~simple, $C_2$-cofinite, rational VOA of CFT type is invariant under
the slash action with weight $0$ of $SL_2(\mathbb{Z})$, and is contained in the space of solutions of a~MLDE (\cite{Zhu}).
The modular invariance property for  $\frac{1}{m}\mathbb{Z}$-graded, simple, $C_2$-cofinite,
$\Z_m$-rational, conical VOAs $V$ is proved in~\cite{V} (see also \cite{DLM}), where
$m\in\Z_{>0}$.
Here a~$\frac{1}{m}\mathbb{Z}$-graded VOA $V=\bigoplus_{n\in\frac{1}{m}
\Z}V_{n}$
is called {\it conical} if $\dim V_0=1$ and $\dim V_{n}=0$ for all $n<0$.
(Therefore, the conical $\Z$-graded VOAs are the same as the VOAs of CFT type.)

Let $W$ be a~$\frac{1}{2}\mathbb{Z}_{\geq 0}$-graded VOA.
Let $W^{(0)}$ and $W^{(1)}$ be the subspaces of $W$ spanned by the homogeneous
vectors of weight $\mathbb{Z}$ and $1/2+\mathbb{Z}$, respectively. 
We say that $W$ is {\em proper} $\frac{1}{2}\mathbb{Z}_{\geq 0}$-graded if $W^{(1)}\neq 0$.
The space spanned by the characters of the modules over a~proper $\frac{1}{2}\mathbb{Z}_{\geq 0}$-graded VOA
is not modular invariant, but the space spanned by the characters of the {\em Ramond-twisted modules}
is modular invariant (see \S~\ref{sec:voa} for the details).

The {\em $\mathcal{W}$-algebras} are $\frac{1}{2}\mathbb{Z}_{\geq 0}$-graded VOAs constructed 
 using the {\em Drinfeld-Sokolov reduction} from affine VOAs (see~\cite{KRW}).
The characters of Ramond-twisted modules over $\mathcal{W}$-algebras 
at admissible levels are determined in~\cite{KW2}.
Moreover, several examples of $\mathcal{W}$-algebras
at non-admissible levels are found to be $C_2$-cofinite and $\Z_2$-rational in~\cite{Kaw}.
Among them, the characters of the Ramond-twisted modules over the $\mathcal{W}$-algebra $\mathcal{W}_{-5}(E_8,f_\theta)$ satisfy the second order MLDE~$(\sharp_{\mu(19/5)})$ (\cite{Kaw}).
It was proved that there do not exist RCFTs whose characters satisfy $(\sharp_{\mu(19/5)})$ (\cite{MMS2}).
Thus we are interested in studying $\mathcal{W}$-algebras and $\frac{1}{2}\mathbb{Z}_{\geq 0}$-graded VOAs with connections of the MLDEs.

In \S~\ref{sec:voa}--\S~\ref{sec:extension}, we show that the characters of the Ramond-twisted modules over the $\mathcal{W}$-algebra
$\mathcal{W}_{-h^\vee/6}(\mathfrak{g},f_\theta)$ at level $-h^\vee/6$ 
associated with a~Lie algebra $\g$ in the Deligne exceptional series
 satisfy~\eqref{eqn:fourth} with 
\begin{equation}\label{eqn:deligne1}
s=\frac{6(7h^\vee-18)}{5(h^\vee+6)},
\end{equation}
 where $h^\vee$ is the dual Coxeter number of $\mathfrak{g}$.
 The $\W$-algebras $W=\W_{-h^\vee/6}(\g,f_\theta)$ satisfy the conditions
 (a)--(h) ({\em Main Conditions})  below:

\medskip\noindent
{\bf Main Conditions:}

 \noindent
 {\rm (a)} There is a~complex number $s$ such that the space of characters of Ramond-twisted $W$-modules is contained
in the space of solutions of~\eqref{eqn:fourth}
and such that any pair of indicial roots of~\eqref{eqn:fourth}
does not have a~non-zero integral difference.
 


\noindent
{\rm (b)} The vector spaces $W^{(0)}$ and $W^{(1)}$ have the forms 
$$
W^{(0)}\cong V\otimes L(-3/5,0),
\quad W^{(1)}\cong P\otimes L(-3/5,3/4)$$
 with a~simple, $C_2$-cofinite, rational VOA $V$
and simple current $V$-module $P$.
Here $L(-3/5,h)$ is the Virasoro minimal model of central charge $-3/5$ and conformal weight $h$.

\noindent
{\rm (c)} The conformal weight of the simple current $V$-module $P$ is non-negative.

\noindent
 {\rm (d)} The conformal weight of each non-vacuum irreducible $V$-module is non-zero.
 
 \noindent
 {\rm (e)} Let $M$ and $N$ be irreducible $V$-modules
with the same conformal weights.
Then the characters of $M$ and $N$ coincide.

\noindent
{\rm (f)}  Let $M$ and $N$ be irreducible Ramond-twisted $W$-modules
with the same conformal weights. Then $M$ and $N$ have the same character.

\noindent
{\rm (g)} Let $S$ and $T$ be irreducible Ramond-twisted $W$-modules with the minimum conformal weight. Let $S\cong\bigoplus_{i=1}^n S^i$ be the irreducible decomposition of $S$ as a~$W^{(0)}$-module.
Then as a~$W^{(0)}$-module, $T$ decomposes into the direct sum $T\cong\bigoplus_{i=1}^n T^i$ such that for each $i=1\ldots,n$, either
$T^i=S^i$ or $T^i=(S^i)'$ holds,
where $(S^i)'$ is the restricted dual of $S^i$ over $W^{(0)}$.

\noindent
{\rm (h)} Any irreducible $V$-module is a~summand of either an~untwisted or Ramond-twisted $W$-module $M$ as a~$V$-module.

\medskip
We show that the proper $\frac{1}{2}\Z_{\geq 0}$-graded VOA $W$ satisfying
(a)--(h) is the $\W$-algebras $\W_{-h^\vee/6}(\g,f_\theta)$ associated with a~Lie algebra $\g$ in the Deligne exceptional series (see Theorem~\ref{sec:identification2}).
For the moment, we cannot omit any of these 8 conditions.
Although condition (a) provides information on the characters of Ramond-twisted $W$-modules as $(\sharp_s)$ gives information on the characters of RCFTs in~\cite{MMS} and \cite{MMS2},
we note that determining a~proper $\frac{1}{2}\mathbb{Z}_{\geq 0}$-graded VOA $W$
by using the characters of Ramond-twisted $W$-modules may not be easy.
It is because the vacuum module of $W$ 
is not Ramond-twisted.
Note also that characters of  minimal $\mathcal{W}$-algebras can be determined by using the Kazhdan-Lusztig polynomials in principle (\cite{A}), but it is not easy to compute the general Kazhdan-Lusztig polynomials.

\medskip
The paper is organized as follows.
In \S~\ref{sec:overview}, we give a~brief explanation of a~proof of Theorem~\ref{sec:identification2} 
by using Main Conditions (the proof is given in \S~\ref{sec:extension}).
In \S~\ref{sec:cft}, candidates of $s\in\mathbb{Q}$ such that~\eqref{eqn:fourth}
has a~solution of CFT type are found by solving Diophantus equations.
Fundamental systems of the space of solutions of such $(\flat_s)$'s
are given in Appendices~\ref{sec:basisrational}--\ref{sec:quasimodular}.
In \S~\ref{sec:construction}, we explain how to construct these solutions of~$(\flat_s)$.
Moreover, we deduce some properties of the space of solutions of~$(\flat_s)$,
which will be used in the proof of Theorem~\ref{sec:identification2}.
In \S~\ref{sec:voa} we briefly recall the notion for VOAs.
In \S~\ref{sec:affine} we review affine VOAs and give a~characterization of
the affine VOAs associated with the Deligne exceptional series in terms of~$(\sharp_s)$.
In \S~\ref{sec:virasoro} we study the modules over an~extension of the minimal model
$L(-3/5,0)$.
In \S~\ref{sec:identification} we show that the characters of the Ramond-twisted 
modules over minimal $\mathcal{W}$-algebras associated with the Deligne exceptional
series satisfy~\eqref{eqn:fourth} for $s$ given by~\eqref{eqn:deligne1}.
Theorem~\ref{sec:identification2} is proved in \S~\ref{sec:extension}.
In \S~\ref{sec:complements} we give several comments on
the relation between the characters of VOAs and~\eqref{eqn:fourth}.
In Appendix~\ref{sec:modularforms}, we introduce modular forms which are necessary to
represent explicit solutions of~\eqref{eqn:fourth} which have solutions of CFT type as explained in Appendices~\ref{sec:basisrational}--\ref{sec:quasimodular}.
We divide such a~set of $s$ into two parts: (I) the number $s$ such that~\eqref{eqn:fourth} has a~solution of CFT type and each solution of CFT type is not quasimodular of positive depth and
(II) the number $s$ such that~\eqref{eqn:fourth} has a~quasimodular solution of CFT type of positive depth.
The quasimodular forms (of positive depth) can be expressed in terms of homogenous polynomials 
that have modular forms as coefficients with respect to~$E_2$~(cf.~\cite{KZ1}).
The solutions of~\eqref{eqn:fourth} of case (I) are studied in
 Appendix~\ref{sec:basisrational}, and that of~\eqref{eqn:fourth} of case (II) are studied
  in Appendix~\ref{sec:quasimodular}.
  In Appendix~\ref{sec:poly}, we define polynomials which are used to represent solutions of MLDEs in Appendices~\ref{sec:basisrational}--\ref{sec:quasimodular}.

\section{Overview of the proof of a~main result}\label{sec:overview}

In this section, we briefly explain how to prove Theorem~\ref{sec:identification2} 
in \S~\ref{sec:extension}.

By condition (a), we have a~restriction on the conformal weights of the Ramond-twisted irreducible $W$-modules.
Then by (b) and (h), we obtain restrictions on the conformal weights of the
 irreducible $V$-modules.
(The condition (h) is related with the orbifold theory (cf.~\cite{CM}).)
By using (a) and (c)--(g), we can study the $V$-modules and
Ramond-twisted $W$-modules.
(The condition (c) is satisfied by all known simple $C_2$-cofinite rational VOAs
to the best of authors' knowledge.)
Suppose that there are only two inequivalent irreducible $V$-modules up to isomorphisms.
It then follows that 
the space of characters of Ramond-twisted $W$-modules is two-dimensional. 
By using the basis of MLDEs given in Appendix~\ref{sec:basisrational}, 
we see that $s=32/5$ and the characters of $V$ satisfies $(\sharp_{\mu(7/2)})$.
It then follows by Theorem~\ref{sec:propsecond} that $V\cong V_1(E_7)$, where $V_k(\mathfrak{g})$ is the affine VOA at level $k$ associated with a~finite-dimensional simple Lie algebra $\mathfrak{g}$.
Then we see that $W\cong \mathcal{W}_{-5}(E_8,f_\theta)$
by Lemma~\ref{sec:branching}.
Suppose that the number of the irreducible $V$-modules up to isomorphisms is more than $2$.
Then we can construct the modules of the VOA $R=V\otimes V_1(A_1)\oplus P\otimes V_1(A_1;\alpha/2)$, and we see that the characters of the $R$-modules satisfy $(\sharp_{\mu(t)})$ with $t\in\mathbb{Q}$.
It then follows by  Theorem~\ref{sec:propsecond}
 that $R\cong V_1(\mathfrak{g})$ with a~Lie algebra $\mathfrak{g}$ in the Deligne exceptional series.
Then we see that $W\cong \mathcal{W}_{-h^\vee/6}(\mathfrak{g},f_\theta)$
by Lemma~\ref{sec:branching}.
(See \S~\ref{sec:extension} for the detail of the proof.)

%

\section{Modular linear differential equations with solutions of CFT type}
\label{sec:cft}

In this section we show that
if $(\flat_s)$ has a~solution of CFT type, then $s$ belongs to~\eqref{eqn:cand5} on \S~\ref{sec:necessarylist}.

\subsection{Indicial equations}
\label{sec:indicial}
Before studying MLDEs with solutions of CFT type, we recall the notion of indicial equations of MLDEs, which is used to classify such MLDEs.

Let $f$ be a~function of the form 
$f(\tau)=q^\alpha\sum_{n=0}^\infty a_nq^n$ with $\alpha\in\mathbb{Q}$, $a_0=1$ and $a_n\in\mathbb{C}$ for each $n\geq 1$.
We call the number $\alpha$ the {\em exponent} of $f$.
Let $n$ be a~non-negative integer, and suppose that~$\alpha+n\ne0$ for any~$n>0$. 
Let $s$ be a~complex number, and
suppose that $f$ is a~solution of~\eqref{eqn:fourth}.
By substituting the $q$-expansion of $f(\tau)$ into~\eqref{eqn:fourth} and comparing the coefficients of $q^{\alpha+n}$ in the both-sides, we obtain a~relation
\begin{align}\label{eqn:recursion1}
&(\alpha+n)^4a_n+\sum_{i=0}^n \biggl(-e_{2,i}(n-i+\alpha)^3a_{n-i}+(3i e_{2,i}+\alpha_1(s)e_{4,i})(n-i+\alpha)^2a_{n-i} \nonumber\\
&\quad-\left(i^2e_{2,i}+\frac{\alpha_1(s)}{2}ie_{4,i}
-\alpha_2(s)e_{6,i}\right)(n-i+\alpha)a_{n-i}+\alpha_3(s)e_{8,i}a_{n-i} \biggr)=0,
\end{align}
where 
$
E_k(\tau)=\sum_{n=0}^\infty e_{k,n}q^n
$
is the Fourier expansion of the normalized Eisenstein series of weight $k$.

By substituting $n=0$ into~\eqref{eqn:recursion1}, we see that $\alpha$ satisfies 
the relation
$(120\alpha-5s-6)(24\alpha-s-6)(24\alpha+s-18)(120\alpha+5s+6)=0$,
which is called the {\em indicial equation} of~\eqref{eqn:fourth}.
The roots $\alpha_1=-s/24-1/20$, $\alpha_2=-s/24+3/4$,
$\alpha_3=s/24+1/4$ 
and $\alpha_4=s/24+1/20$ of the indicial equation are called the {\em indicial roots} ({\em indices}) of~\eqref{eqn:fourth}.
Since $\alpha\in \Q$ and $\alpha$ coincides with one of $\alpha_1,\ldots,\alpha_4$, we have $s\in\Q$.
It is well known that if there is no integral differences among the indicial roots, then for each indicial root $\alpha$, there exists a~solution of~\eqref{eqn:fourth} of the form $f(\tau)=q^\alpha(1+O(q))$.

In the following, suppose that $f$ is of CFT type, so that  $a_n\in\Z_{\geq0}$ for $n\geq 1$.
We call $-24\alpha$ a~{\em formal central charge} of~\eqref{eqn:fourth}.
For each $1\leq i\leq 4$, we call $\alpha_i-\alpha$ a~{\em formal conformal weight}
of~\eqref{eqn:fourth}.
%

\subsection{The case $\mathbf{\alpha=-s/24-1/20}$}\label{sec:alpha1}
It follows from~\eqref{eqn:recursion1} with $n=1$ that
$
(5 s-54) \bigl(25 s^2+5 s a_1+120 s-42 a_1+108\bigr)=0.
$
Assuming that $s\neq 54/5$,
we have
\begin{equation}\label{eqn:diop1}
25s^2+5sa_1+120s-42a_1+108=0.
\end{equation}
By setting $t=5s$, we see that~\eqref{eqn:diop1} becomes 
\begin{equation}\label{eqn:diop12}
t^2+ta_1+24t-42a_1+108=0,
\end{equation}
which is a~monic quadratic in $t$ with integral coefficients.
Therefore, any rational solution $t$ is an~integer.
Since $(t-42)(t+a_1+66)=-2880$ by~\eqref{eqn:diop12} and both $t-42$ and $t+a_1+66$ are integers, we see that $s$ is an~element of the set
\begin{align}\label{eqn:solve1}
\Bigl\{&-\frac{2838}{5},-\frac{1398}{5},-\frac{918}{5},-\frac{678}{5},-\frac{534}{5},-\frac{438}{5},-\frac{318}{5},-\frac{278}{5},-\frac{246}{5},-\frac{198}{5},-30,-\frac{138}{5},-\frac{118}{5},\nonumber\\
&-\frac{102}{5},-\frac{78}{5},-\frac{54}{5},-\frac{48}{5},-\frac{38}{5},-6,-\frac{22}{5},-\frac{18}{5},-\frac{6}{5},-\frac{3}{5},\frac{2}{5},\frac{6}{5},2,\frac{12}{5},\frac{18}{5},\frac{22}{5},\frac{24}{5},\frac{26}{5},\frac{27}{5},6,\frac{32}{5},\\
&\frac{33}{5},\frac{34}{5},\frac{36}{5},\frac{37}{5},\frac{38}{5},\frac{39}{5},8,\frac{41}{5}\Bigr\}.\nonumber
\end{align}

For each value $s$ in~\eqref{eqn:solve1}, we see that any formal conformal weight of~\eqref{eqn:fourth} is not a~positive integer by evaluating the value $\alpha_j-\alpha_1\,(j=2,3,4)$.
Therefore, the Fourier expansion of the solution $f$ is determined by~\eqref{eqn:recursion1}.

By substituting $n=2$ into~\eqref{eqn:recursion1}, we have
\begin{align*}
&-386208-720360s-255500s^2-15750s^3+3125s^4-72792a_1
-22140sa_1-26250s^2a_1\\
&\quad+1375s^3a_1+139536a_2-12960sa_2+300s^2a_2=0.
\end{align*}
Since $a_2$ is a~non-negative integer,~\eqref{eqn:solve1} reduces to
\[
\Bigl\{-\frac{438}{5},-\frac{318}{5},-\frac{198}{5},-30,-\frac{138}{5},-\frac{78}{5},-\frac{48}{5},-\frac{38}{5},-\frac{18}{5},-\frac{6}{5},-\frac{3}{5},\frac{2}{5},\frac{6}{5},\frac{12}{5},\frac{18}{5},\frac{22}{5},\frac{27}{5},6,\frac{32}{5},\frac{39}{5}\Bigr\}.
\]
The similar discussions above for $n=3$ and $4$ show that this set reduces to
\begin{equation}\label{eqn:cand1}
\Bigl\{-\frac{318}{5},-\frac{198}{5},-\frac{138}{5},-\frac{78}{5},-\frac{48}{5},-\frac{38}{5},-\frac{18}{5},-\frac{6}{5},-\frac{3}{5},\frac{2}{5},\frac{6}{5},\frac{12}{5},\frac{18}{5},\frac{22}{5},\frac{27}{5},6,\frac{32}{5}\Bigr\}.
\end{equation}
Hence, if $f$ is a~solution of~\eqref{eqn:fourth} of CFT type,
then $s$ is either $54/5$ or a~value in~\eqref{eqn:cand1}.

\subsection{The case $\mathbf{\alpha=-s/24+3/4}$}
In this case, we have 
$
(s-18) \bigl(75 s^2+15 s a_1+100 s-306 a_1+348\bigr)=0
$
by substituting $n=1$ into~\eqref{eqn:recursion1}.
%
Assuming $s\neq 18$ and setting $t=15s$,
 we have $t^2+3ta_1+20t-918a_1+1044=0$, which is a~monic quadratic equation in $t$ with integral coefficients.
Therefore any rational solution $t$ is an~integer.
Since $(t-306)(t+3a_1+326)=100800$, it follows that $s$ is an~element of the set 
\begin{align}\label{eqn:solve2}
\Bigl\{&-\frac{16698}{5},-\frac{6618}{5},-\frac{4098}{5},-\frac{2298}{5},-\frac{5294}{15},-\frac{1578}{5},-\frac{948}{5},-\frac{858}{5},-\frac{1934}{15},-114,-\frac{498}{5},-\frac{1094}{15},\nonumber\\
&-\frac{318}{5},-\frac{494}{15},-\frac{138}{5},-\frac{254}{15},-\frac{66}{5},-\frac{48}{5},-\frac{44}{15},-\frac{14}{15},-\frac{3}{5},\frac{6}{5},\frac{82}{15},\frac{106}{15},\frac{42}{5},\frac{166}{15},12,\frac{226}{15},\frac{78}{5},\frac{50}{3},\nonumber\\
&\frac{256}{15},\frac{87}{5},\frac{271}{15},\frac{274}{15},\frac{286}{15},\frac{96}{5},\frac{292}{15},\frac{298}{15},\frac{301}{15},\frac{304}{15}\Bigr\}.
\end{align}

For each value $s$ in the above set, we see that any formal conformal weight of~\eqref{eqn:fourth} is not a~positive integer  by evaluating the value $\alpha_j-\alpha_2\,(j=1,3,4)$.
Therefore, the Fourier expansion of the solution $f$ is determined by~\eqref{eqn:recursion1}.

By substituting $n=2$ into~\eqref{eqn:recursion1}, we have
\begin{align*}
&625  s^4-1350  s^3+475 a_1 s^3-155340  s^2-13650 a_1 s^2+140 a_2 s^2-385128  s-1836 a_1 s\\
&\quad-8736 a_2 s-474336 +161352 a_1+136080 a_2=0.
\end{align*}
For an~ $s$ in~\eqref{eqn:solve2},  the following one gives non-negative integer $a_2$
\[
\Bigl\{-\frac{498}{5},-\frac{318}{5},-\frac{138}{5},-\frac{3}{5},\frac{6}{5},\frac{42}{5},\frac{87}{5},\frac{96}{5}\Bigr\}.
\]

Similarly, by considering $n=3,4,\ldots,32$, the set reduces to
\begin{equation}\label{eqn:cand2}
\Bigl\{-\frac{3}{5},\frac{6}{5},\frac{42}{5}\Bigr\}.
\end{equation}
Hence, if $f$ is a~solution of~\eqref{eqn:fourth} of CFT type,
then $s$ is either $18$ or a~value in~\eqref{eqn:cand2}.

\subsection{The case $\mathbf{\alpha=s/24+1/4}$}
It follows from~\eqref{eqn:recursion1} with $n=1$ that 
$
(s+6) \bigl(25 s^2-5 s a_1+130 s-78 a_1+144\bigr)=0.
$
%
%
Assuming $s\neq -6$ and setting $t=5s$, we see that $t^2-ta_1+26t-78a_1+144=0$, which is a~monic quadratic equation in $t$ with integral coefficients.
It then follows that any rational solution $t$ is an~integer.
Since $(t+78)(t-a_1-52)=4200$, we see that $s$ belongs to the set
\begin{align}\label{eqn:solve3}
\Bigl\{&-\frac{77}{5},-\frac{76}{5},-15,-\frac{74}{5},-\frac{73}{5},-\frac{72}{5},-\frac{71}{5},-14,-\frac{68}{5},-\frac{66}{5},-\frac{64}{5},-\frac{63}{5},-\frac{58}{5},-\frac{57}{5},-\frac{54}{5},-\frac{53}{5},\nonumber\\
&-10,-\frac{48}{5},-\frac{43}{5},-\frac{38}{5},-\frac{36}{5},-\frac{28}{5},-\frac{22}{5},-\frac{18}{5},-\frac{8}{5},-\frac{3}{5},\frac{6}{5},\frac{22}{5},\frac{27}{5},\frac{42}{5},\frac{62}{5},\frac{72}{5},18,\frac{97}{5},\frac{122}{5},\frac{132}{5},\nonumber\\
&\frac{202}{5},\frac{222}{5},\frac{272}{5},\frac{342}{5},\frac{447}{5},\frac{522}{5},\frac{622}{5},\frac{762}{5},\frac{972}{5},\frac{1322}{5},\frac{2022}{5},\frac{4122}{5}\Bigr\}.
\end{align}

For each value $s$ in the above set, we see that any formal conformal weight of~\eqref{eqn:fourth} is not a~positive integer  by evaluating the value $\alpha_j-\alpha_3\,(j=1,2,4)$.
Therefore, the Fourier expansion of the solution $f$ is uniquely determined by~\eqref{eqn:recursion1}.

By substituting $n=2$ into~\eqref{eqn:recursion1}, we have
\begin{align*}
&-245592  - 295812 s  - 124830 s^2  - 
  9975 s^3  + 625 s^4  + 15120 a_1 + 9216 s a_1 - 
  6180 s^2 a_1 \\
&\quad- 400 s^3 a_1 + 54648 a_2 + 5016 s a_2 + 
  110 s^2 a_2=0.
\end{align*}
For $s$ in~\eqref{eqn:solve3},  the following one gives non-negative integer $a_2$:
\begin{align*}
\Bigl\{&-15,-\frac{73}{5},-\frac{68}{5},-\frac{66}{5},-\frac{63}{5},-\frac{58}{5},-\frac{54}{5},-\frac{48}{5},-\frac{38}{5},-\frac{18}{5},-\frac{8}{5},-\frac{3}{5},\frac{6}{5},\frac{22}{5},\frac{42}{5},\frac{62}{5},\frac{72}{5},\frac{122}{5},\frac{132}{5},\\
&\frac{222}{5},\frac{342}{5},\frac{447}{5},\frac{762}{5},\frac{2022}{5}\Bigr\}.
\end{align*}
Similarly, by considering $n=3,4,\ldots,23$, the set reduces to
\begin{equation}\label{eqn:cand3}
\Bigl\{-\frac{66}{5},-\frac{18}{5},-\frac{8}{5},-\frac{3}{5},\frac{6}{5}\Bigr\}.
\end{equation}
Hence, if $f$ is a~solution of~\eqref{eqn:fourth} of CFT type,
then $s$ is either $-6$ or a~value in~\eqref{eqn:cand3}.

\subsection{The case $\mathbf{\alpha=s/24+1/20}$}
By substituting $n=1$ into~\eqref{eqn:recursion1}, we have
$
(5 s+66) \bigl(25 s^2-5 s a_1+45 s-18 a_1+18\bigr)=0.
$
%
Assuming $s\neq -66/5$ and setting $t=5s$,
we see that $t^2-ta_1+9t-18a_1+18=0$, which is a~monic quadratic equation in $t$ with integral coefficients.
Therefore the number $t$ is an~integer.
Since $(t+18)(t-a_1-9)=180$, we see that
\begin{align}\label{eqn:solve2}
s=&-\frac{17}{5},-\frac{16}{5},-3,-\frac{14}{5},-\frac{13}{5},-\frac{12}{5},-\frac{9}{5},-\frac{8}{5},-\frac{6}{5},-\frac{3}{5},0,\frac{2}{5},\frac{12}{5},\frac{18}{5},\frac{27}{5},\frac{42}{5},\frac{72}{5},\frac{162}{5}.
\end{align}

For each  $s$ in~\eqref{eqn:solve2}, we see that any formal conformal weight of~\eqref{eqn:fourth}  is not a~positive integer  by evaluating the value $\alpha_j-\alpha_4\,(j=1,2,3)$.
Therefore, the Fourier expansion of the solution $f$ is uniquely determined by~\eqref{eqn:recursion1}.

By substituting $n=2$ into~\eqref{eqn:recursion1}, we have
\begin{align*}
&-661608  - 1138860 s  - 551250 s^2  - 
  47625 s^3  + 3125 s^4  - 41472 a_1 + 11160 s a_1 - 
  24000 s^2 a_1 \\
&\quad- 1750 s^3 a_1 + 176904 a_2 + 18360 s a_2 + 
  450 s^2 a_2=0.
\end{align*}
Then among the values of $s$ in~\eqref{eqn:solve2},  the following give non-negative integer $a_2$:
\[
s=-3,-\frac{8}{5},-\frac{6}{5},-\frac{3}{5},\frac{2}{5},\frac{12}{5},\frac{42}{5}.
\]
Similarly, by considering $n=3$, we see that $s$ is one of the following:
\begin{equation}\label{eqn:cand4}
s=-\frac{8}{5},-\frac{6}{5},-\frac{3}{5},\frac{2}{5},\frac{12}{5},\frac{42}{5}.
\end{equation}
Hence, if $f$ is a~solution of~\eqref{eqn:fourth} of CFT type,
then $s$ is either $-66/5$ or a~value in~\eqref{eqn:cand4}.

\subsection{A~necessary condition of $s$ such that~\eqref{eqn:fourth} has a~solution of CFT type}\label{sec:necessarylist}

Let $s$ be a~complex number and suppose that~\eqref{eqn:fourth} has a~solution
of CFT type.
By combining~\eqref{eqn:cand1},~\eqref{eqn:cand2},~\eqref{eqn:cand3},~\eqref{eqn:cand4} and $s=54/5,18,-6,-66/5$, we see that $s$ is one of the following $23$ numbers:
\begin{align}\label{eqn:cand5}
&-\frac{318}{5},-\frac{198}{5},-\frac{138}{5},-\frac{78}{5},-\frac{48}{5},-\frac{38}{5},-\frac{18}{5},-\frac{6}{5},-\frac{3}{5},\frac{2}{5},\frac{6}{5},\frac{12}{5},\frac{18}{5},\frac{22}{5},\frac{27}{5},6,\frac{32}{5},\frac{54}{5},\\
&\frac{42}{5},
18,-\frac{66}{5},-6,-\frac{8}{5}.\nonumber
\end{align}

\section{Solutions of modular linear differential equations}\label{sec:construction}

Let $s$ be a~number in~\eqref{eqn:cand5}.
In Appendices~\ref{sec:basisrational}--\ref{sec:quasimodular}, we give solutions
of the MLDE $(\flat_s)$ of (quasi)modular forms.
In this section, we explain our method to construct such solutions of~$(\flat_s)$.
We first find a~certain rational number $k$ and a~positive integer $N$, 
and construct solutions $f$ of a MLDE $(\flat_s^k)$ defined below,
assuming that the MLDE $(\flat_s^k)$ has solutions of (quasi)modular forms
of weight $k$ and level $N$.
It then follows that the function $g=\eta^{-2k}f$ is a~solution of $(\flat_s)$. 
Thus, we find solutions of $(\flat_s)$ of quasimodular forms.

Moreover, we give some properties of the space of solutions of~$(\flat_s)$ which is used in
\S~\ref{sec:extension}.

\subsection{Construction of solutions}
In this subsection, we explain our method to construct solutions of MLDEs listed in Appendices~\ref{sec:basisrational}--\ref{sec:quasimodular}.

First, we introduce a~MLDE $(\flat_s^k)$.
Let $k$ and $s$ be complex numbers.
The differential equation
$$
\vartheta^{4}_{k}(f)
+\Big(\alpha_1(s)-\frac{11}{36}\Big)E_4\vartheta^{2}_{k}(f)
+\frac{36\alpha_1(s)+216\alpha_2(s)-5}{216}E_6\vartheta_{k}(f)
+\alpha_3(s)E_8 f\,=\,0
\leqno(\flat_{s}^{k})
$$
is a~MLDE of order 4 with weight~$k$,  
where $\vartheta_{k}(f)=D(f)-k(E_2/12)f$ is the Serre derivation and $\vartheta_k^{i}$ is the iterated 
Serre derivation defined by 
$\vartheta_{k}^{i}(f)=\vartheta_{k+2i-2}\circ\vartheta_{k+2i-4}^{i-1}(f)$ 
 for $i\ge1$ and  $\vartheta^{0}_{*}(f)=f$. 
By setting $k=0$ and using relations~$24\eta'=E_2 \eta$, $12E_2'=E_2^2-E_4$, 
$3E_4'=E_2E_4-E_6$ and $2E_6'=E_2E_6-E_4^2$ (and $E_4^2=E_8$), 
we see that the equation~$(\flat_{s}^{0})$ is equal to the MLDE~$(\flat_{s})$. 
Because $\vartheta_{k}(\eta^{2\ell}\cdot f)=\eta^{2\ell}\vartheta_{k-\ell}(f)$,  
$f$ is a~solution of~$(\flat_{s})$ if and only if $\eta^{2k}\cdot f$ is a~solution of~$(\flat_{s}^{k})$. 

Now we explain how to construct solutions of $(\flat_s)$ of (quasi)modular forms.
Let $s$ be a~number in~\eqref{eqn:cand5} and $\alpha_1,\ldots,\alpha_4$ the indicial roots of $(\flat_s)$ introduced in \S~\ref{sec:indicial}, which are rational numbers.
Set $k=-12\min\{\alpha_1,\ldots,\alpha_4\}$.
Let $N$ be the least common multiple of the denominators of the numbers
$\alpha_1-k/12,\ldots,\alpha_4-k/12$.
Let $F$ be a~solution of $(\flat_s^k)$ and suppose that $F$
is  a~modular form of weight~$k$ and level~$N$.
The modular forms of level $N$ and weight $k$ are homogeneous polynomials in specific modular forms of level $N$ given in Table~\ref{tb:forms} in Appendix~\ref{sec:modularforms}.
Therefore,  by calculating sufficiently many Fourier coefficients of $F$, we can determine $F$ as a~polynomial
in the modular forms.
In order to show that $F$ does satisfy the MLDE~$(\flat_s^k)$, we just
substitute $F$ to the left-hand side of $(\flat_s^k)$ and show that it is zero
by using differential relations and functional equations of modular forms given in Appendix~\ref{sec:modularforms}.
We can also use a~similar method to construct solutions of~$(\flat_s^k)$ of quasimodular forms.

We here perform this method for a~typical case~$\mathbf{s=6/5}$.
Since the set of indicial roots of $(\flat_{6/5})$ is given by $S=\{-1/10,1/10,3/10,7/10\}$,
we have $k=6/5$ and $N=5$.
Let $\alpha$ be any indicial root of $(\flat_{6/5})$.
Let $F_\alpha$ be a~solution of $(\flat_{6/5}^{6/5})$ of the form
$F_\alpha(\tau)=q^{\alpha+6/5}(1+O(q))$.
We suppose that $F_\alpha$ is a~modular form of weight $6/5$
and level $5$.
It then follows that $F_\alpha$ has the form $F_\alpha=\sum_{n=0}^6 \beta_n\psi_1^n\psi_2^{6-n}$ with $\beta_0,\ldots,\beta_6\in\mathbb{C}$, 
where $\psi_1$ and $\psi_2$ are modular forms of weight $1/5$ and level $5$ given
in Table~\ref{tb:forms} in Appendix~\ref{sec:modularforms}:
\begin{eqnarray*}
\psi_1(\tau)=\eta(\tau)^{2/5}q^{-1/60}\prod_{\substack{n>0\\ n\not\equiv 0, \pm 2\ \mathrm{mod}\,5}}\frac{1}{1-q^n}, \quad
\psi_2(\tau)=\eta(\tau)^{2/5}q^{11/60}\prod_{\substack{n>0\\ n\not\equiv 0, \pm 1\ \mathrm{mod}\,5}}\frac{1}{1-q^n}
\end{eqnarray*}
(Rogers-Ramanujan functions).
Since $F_\alpha/\eta^{12/5}$ is a~solution of $(\flat_{6/5})$ of exponent $\alpha$,
we can compute any Fourier coefficient of $F_\alpha$ 
by using~\eqref{eqn:recursion1}.
Therefore we can determine $\beta_0,\ldots,\beta_6$ (by computing finite number of Fourier coefficients of $F_\alpha$.)
In this way, we see that
\begin{align}\label{eqn:falpha}
&F_{-1/10}=\psi_1(\psi_1^5+2\psi_2^5), &&F_{1/10}=\psi_2(2\psi_1^5-\psi_2^5)/2,\nonumber\\
&F_{3/10}=\psi_1^4\psi_2^2, &&F_{7/10}=\psi_1^2\psi_2^4.
\end{align}
Thus we find an~explicit description of $\{F_\alpha|\alpha\in S\}$ 
 under the assumption that
$F_\alpha$ is a~modular form of weight $6/5$ and level $5$ for any $\alpha\in S$.
By using the differential relations of $\psi_1$ and $\psi_2$ written in  Appendix~\ref{sec:modularforms} {\bf(d)}, we can show that
the right-hand sides of~\eqref{eqn:falpha} do satisfy the MLDE~$(\flat_{6/5}^{6/5})$.
Since $\{f_{\alpha+1/10}:=F_\alpha/\eta^{12/5}|\alpha\in S\}$
is a~fundamental system of solutions of~$(\flat_{6/5})$,
we have an~explicit description of solutions of $(\flat_{6/5})$ in terms of $\psi_1$
and $\psi_2$, as written in Appendix~\ref{sec:basisrational} {\bf(f)}.

In a~similar way, we can construct solutions of $(\flat_s)$ listed in Appendices~\ref{sec:basisrational}--\ref{sec:quasimodular}.
By using the list of solutions, we have the following theorem.

\begin{theorem}\label{sec:modularcases}
Let $s$ be a~complex number and suppose that $(\flat_s)$ has a~solution
$f$ of CFT type. Suppose that $f$ is not a~quasimodular form of positive depth.
Then $s$ is one of the following $17$ numbers:
\begin{equation}\label{eqn:cand6}
-\frac{48}{5},-\frac{38}{5},-\frac{6}{5},-\frac{3}{5},\frac{2}{5},\frac{6}{5},\frac{12}{5},\frac{18}{5},\frac{22}{5},\frac{27}{5},6,\frac{32}{5},\frac{54}{5},18,-\frac{66}{5},-6,-\frac{8}{5}.
\end{equation}
\end{theorem}

\subsection{Modular invariant subspaces}
In this subsection we find some properties of the space of solutions of~$(\flat_s)$ which is used in
\S~\ref{sec:extension}.

A~function $f$ on $\mathbb{H}$ is called of {\em character type} if $f$ has the form
$f(\tau)=q^\beta\sum_{n=0}^\infty a_n q^n$ with $\beta\in \mathbb{Q}$,
$a_n\in\mathbb{Z}_{\geq 0}$ ($n\in\Z_{\geq0}$) and $a_0\geq 1$.

\begin{proposition}\label{sec:modularinv1}
Let $s$ be one of the numbers $-3/5,2/5,6/5,12/5,18/5,22/5,27/5,6$ and $32/5$.
Let $f$ and $g$ be linearly-independent solutions of~\eqref{eqn:fourth}
with exponents $\alpha$ and $\beta\in\mathbb{Q}$. 
Suppose that $\beta-\alpha=4/5$, $f$ is of CFT type, $g$ is of character type,
and the space spanned by $f$ and $g$ is modular invariant. 
Then $s=32/5$.
\end{proposition}

\begin{proof}
We show the assertion by contradiction in case by case bases.
Suppose that $s=-3/5$.
Then the set of indices of $(\flat_{-3/5})$ is $\{31/40,9/40,1/40,-1/40\}$.
Since $\beta-\alpha=4/5$, we have $\alpha=-1/40$ and $\beta=31/40$.
Therefore, $f=f_0$ and  $g=c\cdot f_{4/5}$ with $c\in\mathbb{C}^\times$, where the solutions $f_h$
with  formal conformal weights $h$ of $(\flat_{-3/5})$ are defined in Appendix~\ref{sec:basisrational} {\bf(d)}.
By using the expression of $f_0$ in Appendix~\ref{sec:basisrational} {\bf(d)}, we can show that $f_0(-1/\tau)=\sqrt{2/5}\bigl(\sin(2\pi/5)f_0(\tau)+\sin(\pi/5)f_{4/5}(\tau)
+\sin(\pi/5)f_{1/20}(\tau)+\sin(2\pi/5)f_{1/4}(\tau)\bigr)$.
Therefore, the space spanned by $f_0$ and $f_{4/5}$ is not closed under $SL_2(\Z)$, which contradicts the assumption.
We can reduce the other cases outside $s=32/5$ to contradiction in a~similar way.
Thus we have $s=32/5$.
\end{proof}

The {\em modular Wronskian} $W(F)$ of a~vector-valued modular function
$F={}^t(f_1,\ldots,f_n)$ ($n\geq 1$) is defined by
\[
W(F)=\mathrm{det}(F,\vartheta_0(F),\vartheta_0^2(F),\ldots,\vartheta_0^{n-1}(F)).
\]

\begin{proposition}\label{sec:modularinv2}
If either $s=-48/5$ or $-38/5$, then there is no modular invariant subspace $U$ of the space of solutions of~{\rm $(\flat_s)$}
such that $U$ is spanned by functions of character type and such that a~function of CFT type belongs to $U$.
\end{proposition}

\begin{proof}
We here only show the case of $s=-48/5$.
The case of $s=-38/5$ is similarly proved.
Since $f_0,f_{4/5},f_{-1/2}$ and $f_{-7/10}$ defined in {\bf (a)} in Appendix~\ref{sec:basisrational} form a~fundamental system of solutions of~$(\flat_{-48/5})$, it follows that
\begin{equation*}
f_0(-1/\tau)\=c_1 f_0(\tau)+c_2 f_{4/5}(\tau)+c_3f_{-1/2}(\tau)+c_4f_{-7/10}(\tau)
\qquad (c_i\in\mathbb{C}).
\end{equation*}
As any scalar multiple of $f_{-7/10}$ is not of character type, it suffices to show $c_4\neq 0$.
Suppose that $c_4=0$.
Since $\eta^{42/5}f_0$ is not a~modular form (with a~multiplier system)
 on $SL_2(\mathbb{Z})$,
one of $c_2$ and $c_3$ is non-zero.
Therefore, one of
\begin{eqnarray*}
F_1:=\begin{pmatrix}f_0\\ f_{4/5}\\ f_{-1/2}\end{pmatrix}\,,\quad 
F_2:=\begin{pmatrix}f_0\\ f_{-1/2}\end{pmatrix}\,,\quad 
F_3:=\begin{pmatrix}f_0\\ f_{4/5}\end{pmatrix}
\end{eqnarray*}
is a~vector-valued modular function on $SL_2(\mathbb{Z})$.
Let $\lambda_i$ be the sum of exponents of all components of the vector $F_i$.
Since the exponents of $f_0,f_{4/5}$ and $f_{-1/2}$ are $7/20,23/20$ and $-3/20$, respectively, it follows that $\lambda_1=27/20$, $\lambda_2=1/5$ and  $\lambda_3=3/2$.
Suppose that $F_i$ is a~vector-valued modular function on $SL_2(\mathbb{Z})$.
It then follows from~\cite[Theorem~3.7]{Mas} that 
$W(F_i)/\eta^{24\lambda_i}$ is a~non-zero classical modular form on $SL_2(\mathbb{Z})$ of weight $w=6-12\lambda_i$.
Then we see that $w$ is negative, it contradicts to the fact that there does not
exist a~non-zero classical modular form on $SL_2(\mathbb{Z})$ of negative weight.
Hence, we have $c_4\neq 0$, as desired.
\end{proof}

\section{Vertex operator algebras}\label{sec:voa}

In this section we recall some properties for $\frac{1}{2}\Z_{\geq 0}$-graded VOAs.

Let $W=\bigoplus_{n\in\frac{1}{2}\mathbb{Z}_{\geq 0}}W_n$ be a~simple $\frac{1}{2}\mathbb{Z}_{\geq 0}$-graded conical VOA with central charge $c\in \C$.
Set $W^{(0)}=\bigoplus_{n\in\mathbb{Z}}W_n$ and $W^{(1)}=\bigoplus_{n\in1/2+\mathbb{Z}}W_n$.
Then there exists an~automorphism $\sigma$ of $W$ defined by
$\sigma(a)=a$ and $\sigma(v)=-v$ for each $a\in W^{(0)}$ and $v\in W^{(1)}$.
If $W^{(1)}\neq 0$, then the order of $\sigma$ is $2$,
otherwise $\sigma=\mathrm{id}_{W}$. 
We set $G=\langle \sigma \rangle$.
The VOA $W$ is called {\em $\Z_2$-rational} if $W$ is rational and $\sigma$-rational.
The $W$-module $W$ is called the {\em vacuum module}.

Suppose that $W$ is $C_2$-cofinite and $\Z_2$-rational.
The $\sigma$-twisted $W$-modules are called the {\em Ramond-twisted modules}.
It then follows that any Ramond-twisted irreducible $W$-module $M$ has the form $M=\bigoplus_{n=0}^\infty M_{n+h}$, where $h\in\mathbb{C}$ and
$M_{n+h}=\{v\in M\,|\,L_0v=(n+h)v\}$ such that $M_h\neq 0$.
The number $h$ is called the {\em conformal weight} of $M$.
If $W^{(1)}=0$, then $W$ is a~$\mathbb{Z}_{\geq 0}$-graded VOA, and then
the Ramond-twisted $W$-modules coincide with the untwisted $W$-modules.
The space spanned by the characters of the Ramond-twisted modules
is $SL_2(\mathbb{Z})$-invariant (see \cite[Theorem 1.3]{V}).

Let $M$ be a~Ramond-twisted irreducible $W$-module with a~conformal weight $h$.
The {\em character} of $M$ is defined by $\chi_M(\tau)=q^{h-c/24}\sum_{n=0}^\infty\dim M_{n+h}\,q^n$.
The exponent of $\chi_M$ is also called an~exponent of $M$.
Suppose that the conformal weight of any Ramond-twisted irreducible $W$-module is a~rational number.
The minimum number $h_{\mathrm{min}}$ among the conformal weights of all 
Ramond-twisted irreducible $W$-modules is called the {\em minimum conformal weight} of 
$W$.
The number $\tilde{c}=c-24h_{\mathrm{min}}$ is called the {\em effective central charge} of $\frac{1}{2}\Z_{\geq 0}$-graded VOA $W$ (\cite{DM1}).
For a~proof of the following lemma, refer to~\cite{Kaw}.

\begin{lemma}\label{sec:lemsuper}
Let $W$ be a~vertex operator superalgebra, $V$ a~vertex operator subalgebra
 of $W$ and suppose that $W=V\oplus P$ with a~simple current $V$-module $P$.
Let $w$ be an element of $P$ and $a$ an~integer such that 
$z^a Y(w,z)w\in W[[z]]$ and $z^aY(w,z)w|_{z=0}\neq 0$.
If $a\in2\mathbb{Z}$, then $W$ is a~vertex operator algebra.
\end{lemma}

\section{Affine vertex operator algebras and Virasoro minimal models}\label{sec:affinevirasoro}

In this section we recall the notion of affine VOAs and give a~characterization of
the Deligne exceptional series by using the MLDEs $(\sharp_s)$.
Moreover, we study extensions of the Virasoro minimal model $L(-3/5,0)$, which will be used to study $\mathcal{W}$-algebras.

\subsection{Affine vertex operator algebras and second order modular linear differential equations}\label{sec:affine}

Let $\mathfrak{g}$ be a~finite-dimensional simple Lie algebra
and $\widehat{\mathfrak{g}}$ the affine Kac-Moody Lie algebra associated with $\mathfrak{g}$.
The simple affine VOA at level $k\in\mathbb{C}$ associated with $\mathfrak{g}$
is denoted by $V_k(\mathfrak{g})$.
Let $\lambda$ be a~weight of $\mathfrak{g}$.
If the simple $\widehat{\mathfrak{g}}$-module $L(\lambda+k\Lambda_0)$ is a~$V_k(\mathfrak{g})$-module, we write $V_k(\mathfrak{g};\lambda)=L(\lambda+k\Lambda_0)$.
Recall the MLDEs of second order $(\sharp_s)$ 
and the number
$\mu(t)=t(t+2)/144$  defined in Introduction.

Let $V$ be a~$C_2$-cofinite, rational VOA of CFT type with the central charge $c\in\C$
such that the characters of the modules over $V$ satisfy~$(\sharp_s)$
and that the conformal weights of the irreducible $V$-modules are positive.
In~\cite{MMS}, it is shown that the character of $V$ coincides with
the character of $V_1(\g)$ with a~Lie algebra $\g$ in the Deligne
exceptional series and $s=\mu(c/2)$.
Moreover, the unitary RCFTs whose characters satisfy~$(\sharp_s)$ are classified in~\cite{MMS2}: they are the WZW models at level $1$ associated to
the Deligne exceptional series.
We here prove a~vertex operator algebraic analogue of this result.

\begin{theorem}\label{sec:propsecond}
Let $V=\bigoplus_{n=0}^\infty V_n$ be a~self-dual vertex operator algebra
with the central charge $c\in\C$. 
Suppose that $V$ satisfies the following 3 conditions\,{\rm:}
{\rm(M1)} $V$ is simple, $C_2$-cofinite, rational and of CFT type, 
{\rm(M2)} the characters of the modules over $V$ satisfy modular linear differential 
equation~$(\sharp_s)$ of order $2$ and weight $0$ for some $s\in\C$,
{\rm(M3)} $\dim V_1\geq 3$.
Then $V$ is isomorphic to the simple affine vertex operator algebra
at level $1$ associated with
a~Lie algebra in the Deligne exceptional series and $s=\mu(c/2)$.
\end{theorem}

As we shall see in Remark~\ref{sec:mms} below, we can replace the condition (M3) in the theorem with the condition 
(M4): the conformal weights of all irreducible modules over $V$ are non-negative.
Note that the conformal weights of the non-vacuum irreducible modules over the chiral algebra (VOA)
of a~unitary RCFT are positive.
In order to prove the theorem, we first show a~lemma.

\begin{lemma}\label{sec:lemindex}
Let $L(f)=0$ be a~modular linear differential equation of order $n\leq 4$ and weight $0$.
Let $\alpha_1,\ldots,\alpha_n$ be all indices of $L(f)=0$ and suppose that $\alpha_i\in\mathbb{Q}$ 
and $\alpha_i<\alpha_j$ for $1\leq i\neq j\leq n$.
Let $f_1,\ldots,f_n$ be holomorphic functions on $\mathbb{H}$ such that $f_i$ has the form
$f_i(\tau)=q^{\alpha_i}(1+O(q))$ for $1\leq i\leq n$, and that $F={}^t(f_1,\ldots,f_n)$ forms 
a~vector-valued modular form of weight~$0$. 
Then $(f_1,\ldots,f_n)$ is a~fundamental system of solutions of $L(f)=0$.
\end{lemma}

\begin{proof}
Since $F$ forms a~vector-valued modular form of weight~$0$, we have a~non-zero classical 
modular form~$G$ of weight~$n^2-n-12\sum_{i=1}^n \alpha_i$ such that 
$W(F)=G\cdot \eta^{24\sum_{i=1}^n \alpha_i}$ (\cite[Theorem~3.7]{Mas})\,.
By evaluating the coefficient of $E_2D^{n-1}(f)$ in~$L(f)=0$, which is equal to $-n(n-1)/12$, 
we find that $n(n-1)/12$ is equal to the sum of all exponents.
Thus we find $n^2-n-12\sum_{i=1}^n \alpha_i=0$ so that $G$ is a~non-zero constant. 
Therefore $(f_1,\ldots,f_n)$ is a~fundamental system of solutions of~$L(f)=0$ by~\cite[Theorem~4.3]{Mas}\,.
\end{proof}

Recall that the indices of $(\sharp_{\mu(t)})$ are given by $-t/12$ and
$(t+2)/12$ (\cite{KNS}).
We now prove Theorem~\ref{sec:propsecond}.

\begin{proof}[Proof of Theorem~\ref{sec:propsecond}]
Put $d=\dim\,V_1$ and let $\tilde{c}$ be the effective central charge of $V$.
By using (M1) and (M2), we see that 
$(\tilde{c},d)=(2/5,0)$, $(2/5,1)$, $(1,1)$, $(1,3)$,
$(2,3)$, $(2,8)$, $(14/5,14)$, $(4,28)$, $(4,8)$, $(26/5,52)$, $(6,78)$, $(6,14)$,
$(7,133)$, $(38/5,190)$ or
$(8,248)$,
and $s=\mu(\tilde{c}/2)$ (\cite{MMS} and~\cite{KNS}).
The assumption (M3) forces $(\tilde{c},d)\neq (2/5,0)$, $(2/5,1)$ and $(1,1)$.
Since $V_1$ is a~reductive Lie algebra  and $\mathrm{rank}\,V_1\leq \tilde{c}$
by \cite[Theorem~1.1 and Theorem~1.2]{DM1},
we see that $V_1$ is one of the Lie algebras listed in Table~\ref{tb:table2}.

\begin{table}[hb]
\caption{Possibility of $\g$.}\label{tb:table2}
\begin{tabular}{|c|c|c|c|c|c|c|c|c|c|}
\hline
$(\tilde{c},d)$ & & $(2,8)$ & $(14/5,14)$ & & $(4,8)$ & & & &  \\
$d$ & 3 & & &28&&52&78&133&248\\
$V_1$ & $A_1$ & $A_2$ & $G_2$ & $G_2^2$, $D_4$ & $A_2$, $A_1^2\times \mathbb{C}^2$ & $F_4$ & $B_6$, $C_6$, $E_6$ & $E_7$ &$E_8$\\
\hline
\end{tabular}
\end{table}

It follows from~\cite[Theorem 1.1]{DM2} that the tensor product affine VOA $U=\bigotimes _{i=1}^n V_{k_i}(\mathfrak{g}_i)$ with 
$k_1,\ldots,k_n\in\mathbb{Z}_{>0}$ is embedded in $V$, where 
$n\in\Z_{>0}$ and $\g_1,\ldots,\g_n$ are simple ideals of $\g$ such that $\g$ decomposes into the sum $\g=\mathfrak{a}\oplus\bigoplus_{i=1}^n\g$ with an~abelian Lie algebra $\mathfrak{a}$.
By comparing $\dim\,V_2$ with $\dim\,U_2$,
we see that $n=1$, $k_1=1$ and $\mathfrak{g}$ is a~Lie algebra in the Deligne exceptional series.
(The dimension of $V_2$ can be computed by using the
MLDE $(\sharp_{\mu(\tilde{c}/2)}$).
We here only give a~proof of the case $d=28$.
The other cases are proved in a~similar way.

Since $\chi_V(\tau)$ is a~solution of CFT type of $(\sharp_{\mu(2)})$ and $\dim\,V_1=28$,
it follows that $\chi_V(\tau)=q^{\alpha}(1+28q+134q^2+O(q^3))$ with $\alpha\in\mathbb{Q}$,
and hence we have $\dim\,V_2=134$.
Suppose that $V_1\cong G_2^2$.
It then follows that $U\cong V_{k_1}(G_2)\otimes V_{k_2}(G_2)$
with $k_1,k_2\in \mathbb{Z}_{>0}$.
We see that $\dim\,U_2=280$ if $k_1=k_2=1$,
$\dim\,U_2=434$ if $k_1\neq 1$ and $k_2\neq 1$
and $\dim\,U_2=357$ otherwise,
which contradicts $\dim\,V_2=134$.
Hence, $V_1\cong D_4$.
It then follows that $U\cong V_k(D_4)$ with $k\in\mathbb{Z}_{>0}$.
Then we have $k=1$ since $\dim\,U_2=434$ otherwise.
Hence $V_1(D_4)$ is embedded in $V$.
Since the space of the characters of $V_1(D_4)$-modules is $2$-dimensional
and the exponents of $V_1(D_4)$-modules are $-1/6$ and 
$1/3$, it follows that the characters of $V_1(D_4)$
satisfy $(\sharp_{\mu(2)})$ by Lemma~\ref{sec:lemindex}. Therefore, as $V_1(D_4)\subset V$, we have
$V\cong V_1(D_4)$, and since the central charge of $V_1(D_4)$ is $4$, we see that $s=\mu(c/2)$.
\end{proof}

\begin{remark}\label{sec:mms}
(i) In~\cite{MMS} and~\cite{MMS2}, it is also proved that
a~RCFT whose characters satisfy~$(\sharp_{\mu(1/5)})$ coincides with 
the Virasoro minimal model of central charge $-22/5$.

\medskip\noindent
(ii) In Theorem~\ref{sec:propsecond}, we can replace condition (M3)  with (M4).
(The condition (M4) is stated just after Theorem~\ref{sec:propsecond}.)
It is because in the proof of Theorem~\ref{sec:propsecond}, we use 
(M3) in order to exclude the cases $(\tilde{c},d)=(2/5,0)$, $(2/5,1)$ and
$(1,1)$.
However, (M4) also excludes the cases of $(\tilde{c},d)=(2/5,0)$ and
$(1,1)$.
Since $\tilde{c}\geq \mathrm{rank}(V_1)$, we also have $(\tilde{c},d)\neq(2/5,1)$.
Hence Theorem~\ref{sec:propsecond} holds even if we replace (M3) with (M4).
\end{remark}

\subsection{Extensions of a~Virasoro minimal model}\label{sec:virasoro}

We denote by $L(-3/5,0)$ the Virasoro minimal model of central charge $-3/5$.
In this subsection we study extensions of $L(-3/5,0)$, which will be used to study
$\W$-algebras in the subsequent sections \S~\ref{sec:identification} and
\S~\ref{sec:extension}.

All irreducible $L(-3/5,0)$-modules up to the isomorphisms are given by
$
L(-3/5,0)$, $L(-3/5,-1/20)$, $L(-3/5,3/4)$ and $L(-3/5,1/5)$.
The module $L(-3/5,3/4)$ is a~simple current with the fusion products
$L(-3/5,3/4)\boxtimes L(-3/5,3/4)=L(-3/5,0)$
and $L(-3/5,3/4)\boxtimes L(-3/5,-1/20)=L(-3/5,1/5)$.

Let $W$ be a~$\frac{1}{2}\mathbb{Z}_{\geq 0}$-graded VOA of the form
$W^{(0)}=V\otimes L(-3/5,0)$ and $W^{(1)}=P\otimes L(-3/5,0)$
with a~simple, rational, $C_2$-cofinite VOA $V$ and a~simple current $V$-module $P$.
It then follows that $P\boxtimes P=V$.
Let $M$ be an~irreducible $V$-module and 
$h$ an~element of $\{0,-1/20,3/4,1/5\}$.
\begin{definition}\label{sec:def1}
Define the $V\otimes L(-3/5,0)$-modules $L(M;h)$ by
\[
L(M;h)=M\otimes L(-3/5,h)\oplus (M\boxtimes P)\otimes (L(-3/5,h)\boxtimes L(-3/5,3/4)).
\]
\end{definition}

\begin{lemma}\label{sec:3414}
Suppose that $h_{M\boxtimes P}-h_{M}\in 3/4+\mathbb{Z}$ {\rm(}resp., $1/4+\mathbb{Z}${\rm)}. Then the following hold.
{\rm{\bf(1)}}  $L(M;-1/20)$ and $L(M;3/4)$ are irreducible
Ramond-twisted $W$-modules {\rm(}resp., untwisted $W$-modules{\rm)}.
{\rm{\bf(2)}} $L(M\boxtimes P;-1/20)$ and $L(M\boxtimes P;3/4)$ are irreducible untwisted $W$-modules {\rm(}resp., Ramond-twisted $W$-modules{\rm)}.
\end{lemma}

\begin{proof}
{\rm{\bf(1)}} Suppose $h_{M\boxtimes P}-h_{M}\in 3/4+\mathbb{Z}$.
We show that $L(M;-1/20)$ is an~irreducible Ramond-twisted $W$-module.
The other statements can be proved in a similar way.
We have $L(M;-1/20)= M\otimes L(-3/5,-1/20)\oplus (M\boxtimes P)\otimes
L(-3/5,1/5)$.
The conformal weight of the first summand is $h_M-1/20$, and that of the 
second summand is $h_{M\boxtimes P}+1/5\in h_M+19/20+\mathbb{Z}$.
Hence, the conformal weights of $L(M;-1/20)$ lie in $h_M-1/20+\mathbb{Z}$.
Therefore, $L(M;-1/20)$ is Ramond-twisted.
The irreducibility follows from the the construction of $L(M;h)$.

\noindent{\rm{\bf(2)}} Since $(M\boxtimes P)\boxtimes P\cong M$ by the associativity of the fusion products and $P\boxtimes P\cong V$, the statement follows from {\rm{\bf(1)}}.
\end{proof}

\section{Identification and $\mathcal{W}$-algebras}\label{sec:identification}

In this section we show the following theorem.

\begin{theorem}\label{sec:identification1}
If $\mathfrak{g}$ is a~Lie algebra in the Deligne exceptional series,
then the characters of the Ramond-twisted modules over the $\mathcal{W}$-algebra $\mathcal{W}_{-h^\vee/6}(\mathfrak{g},f_\theta)$ satisfy~\eqref{eqn:fourth}, where $s$ is given by~\eqref{eqn:deligne1}.
\end{theorem}

We first introduce some notations and show a lemma, which is needed to give a~proof of Theorem~\ref{sec:identification1}.

Let $\mathfrak{g}$ be a~Lie algebra in the Deligne exceptional series
and $W(\mathfrak{g})=\mathcal{W}_{-h^\vee/6}(\mathfrak{g},f_\theta)$ 
the simple $\mathcal{W}$-algebra at level $k=-h^\vee/6$ associated with
$\mathfrak{g}$ and a~minimal nilpotent element $f_\theta$ of $\mathfrak{g}$.
If $\mathfrak{g}=A_1$, then it is well known that $W(A_1)\cong L(-3/5,0)$ (see e.g.~\cite{KRW}).
Therefore, we see that the characters of $W(A_1)$ satisfy~$(\flat_{-3/5})$.

We now suppose that $\mathfrak{g}\neq A_1$.
Then $W(\mathfrak{g})$ is a~proper $\frac{1}{2}\mathbb{Z}_{\geq 0}$-graded conical VOA.
Define a~VOA $U_\mathfrak{g}$ and simple current $V$-module $P_\mathfrak{g}$ by the formula
$V_1(\mathfrak{g})\cong U_\mathfrak{g}\otimes V_1(A_1{(\theta)})\oplus P_\mathfrak{g}\otimes V_1(A_1{(\theta)};\theta/2)$,
where $\theta$ is the highest root of $\mathfrak{g}$ and $A_1{(\theta)}$ is 
the simple Lie algebra of type $A_1$ generated by the highest and lowest root vectors of $\mathfrak{g}$.

\begin{lemma}\cite{Kaw}\label{sec:branching}
The vertex operator algebra $W(\mathfrak{g})$ is $C_2$-cofinite and $\Z_2$-rational. Moreover, $W(\mathfrak{g})$ is isomorphic to the simple current extension of $U_\mathfrak{g}\otimes L(-3/5,0)$ by $P_\mathfrak{g}\otimes L(-3/5,3/4)$.
\end{lemma}

We denote by $\mathrm{Mod}$ and $\mathrm{Mod}_{\mathrm{tw}}$ the sets of the equivalence classes of the irreducible untwisted and Ramond-twisted modules over $W(\mathfrak{g})$, respectively.
Let $S$ be the space of characters of Ramond-twisted  $W(\mathfrak{g})$-modules
and $T$ the space of characters of $U_\mathfrak{g}$-modules.

\begin{proof}[{\bf Proof of Theorem~\ref{sec:identification1}}]
For $\g=A_2,G_2,D_4,F_4,E_6,E_7$ and $E_8$, the numbers $s$ in~\eqref{eqn:deligne1} are $s=2/5,6/5,12/5,18/5,22/5,27/5$ and $32/5$, respectively.
We show the assertion in a~case by case basis.

\medskip\noindent{{\bf 1. $\mathfrak{g}=A_2$.}}
Since $U_{A_2}=V_{\sqrt{3}A_1}$ and $P_{A_2}=V_{\sqrt{3}A_1+\sqrt{3}\alpha/2}$,
the complete set of the irreducible modules over $V_{\sqrt{3}A_1}$ up to isomorphisms is given by
$
N_k:=V_{\sqrt{3}A_1+k\sqrt{3}\alpha/6}$ with $k\in\mathbb{Z}_6$.
The conformal weights of $N_i$ ($0\leq i\leq 5$) are $0,1/12,1/3,3/4,1/3$ and $1/12$, respectively.
The space $T$ has a~basis
$(\chi_{N_i})_{0\leq i\leq 3}$.
Since  $P_{A_2}\boxtimes N_k=N_{k+3}$ for each $k\in\mathbb{Z}_6$, 
it follows by Lemma~\ref{sec:3414} that
$
\mathrm{Mod}=\{L(N_k;-1/20),L(N_k;3/4)\,|\,k\in\mathbb{Z}_6,k\mbox{ is odd}\}$ and
$\mathrm{Mod}_{\mathrm{tw}}=\{L(N_k;-1/20),L(N_k;3/4)\,|\,k\in\mathbb{Z}_6,k\mbox{ is even}\}.
$
The space $S$  has a~basis
$\chi_{L(N_0;-1/20)},\chi_{L(N_4;3/4)},\chi_{L(N_2;-1/20)},\chi_{L(N_0;3/4)}$ with the conformal weights $-1/20,1/12,17/60,3/4$, respectively.
Since the central charge of $W(A_2)$ is $2/5$, the exponents of this basis of $S$ are $-1/15, 1/15, 4/15, 11/15$.
As they correspond to the indices of MLDE~$(\flat_{2/5})$, we see that
each element of $S$ satisfies~$(\flat_{2/5})$  by Lemma~\ref{sec:lemindex}.

\medskip\noindent{{\bf 2. $\mathfrak{g}=G_2$.}}
Since $U_{G_2}=V_3(A_1)$ and $P_{G_2}=V_3(A_1;3\alpha/2)$,
the complete set of the irreducible modules over $V_3(A_1)$ up to the isomorphisms is 
given by
$
N_k:=V_3(A_1;k\alpha/2)$ with $k\in\mathbb{Z}_4.
$
The conformal weights of $N_i$ $(0\leq i\leq 3)$ are $0,3/20,2/5$ and $3/4$.
Since $P_{G_2}\boxtimes  N_k=N_{3-k}$ for each $k\in\mathbb{Z}_4$, 
it follows that
$
\mathrm{Mod}=\{L(N_k;-1/20),L(N_k;3/4)\,|\,k\in\mathbb{Z}_4,k\mbox{ is odd}\}$ and
$\mathrm{Mod}_{\mathrm{tw}}=\{L(N_k;-1/20),L(N_k;3/4)\,|\,k\in\mathbb{Z}_4,k\mbox{ is even}\}.
$
The space $S$  has a~basis
$\chi_{L(N_0;-1/20)}$, $\chi_{L(N_2;3/4)}$, $\chi_{L(N_2;-1/20)}$, 
$\chi_{L(N_0;3/4)}$ with the conformal weights $-1/20,3/20,7/20,3/4$, respectively.
Since the central charge of $W(G_2)$ is $6/5$, the exponents of the basis of $S$ are $-1/10, 1/10, 3/10, 7/10$, which shows that
each element of $S$ satisfies~$(\flat_{6/5})$  by Lemma~\ref{sec:lemindex}.

\medskip\noindent{{\bf 3. $\mathfrak{g}=D_4$.}}
Since $U_{D_4}=V_1(A_1)^{\otimes 3}$ and $P_{D_4}=V_1(A_1;\alpha/2)^{\otimes 3}$,
the complete set of the irreducible modules over $V_1(A_1)^{\otimes 3}$ up to the isomorphisms is given by
$
N_{k_1,k_2,k_3}:=V_1(A_1;k_1\alpha/2)\otimes V_1(A_1;k_2\alpha/2)\otimes V_1(A_1;k_3\alpha/2)$ with $k_1,k_2,k_3\in \mathbb{Z}_2.
$
The conformal weight of $N_{k_1,k_2,k_3}$ is $(k_1^2+k_2^2+k_3^2)/4$.
The space $T$ has a~basis $\chi_{N_{0,0,0}},\chi_{N_{1,0,0}},\chi_{N_{1,1,0}},\chi_{N_{1,1,1}}$.
Since  $P_{D_4}\boxtimes  N_{k_1,k_2,k_3}=N_{1-k_1,1-k_2,1-k_3}$ for each $k_1,k_2,k_3\in\mathbb{Z}_2$,  
it follows that
$
\mathrm{Mod}=\{L(N_{k_1,k_2,k_3};-1/20),
L(N_{k_1,k_2,k_3};3/4),\,|\,k_1,k_2,k_3\in\mathbb{Z}_2,k_1+k_2+k_3\mbox{ is odd}\}$ and
$\mathrm{Mod}_{\mathrm{tw}}=\{L(N_{k_1,k_2,k_3};-1/20),
L(N_{k_1,k_2,k_3};3/4),\,|\,k_1,k_2,k_3\in\mathbb{Z}_2,k_1+k_2+k_3\mbox{ is even}\}.
$
The space $S$  has a~basis
$\chi_{L(N_{0,0,0};-1/20)}$, $\chi_{L(N_{0,1,1};3/4)}$, $\chi_{L(N_{1,1,0};-1/20)}$, $\chi_{L(N_{0,0,0};3/4)}$ with the conformal weights $-1/20,1/4,9/20,3/4$, respectively.
Since the central charge of $W(D_4)$ is $12/5$, the exponents of the basis of $S$ are $-3/20, 3/20, 7/20, 13/20$, which shows that
each element of $S$ satisfies~$(\flat_{12/5})$ by Lemma~\ref{sec:lemindex}.

\medskip\noindent{{\bf 4. $\mathfrak{g}=F_4$.}}
Since $U_{F_4}=V_1(C_3)$ and $P_{F_4}=V_1(C_3;\varpi_3)$,
the complete set of the irreducible modules over $V_1(C_3)$ is 
given by
$
N_k:=V_1(C_3;\varpi_k)$ with $k\in\mathbb{Z}_4,
$ 
where $\varpi_0=0$.
The conformal weights of $N_k$ $(0\leq k\leq 3)$ are $0,7/20,3/5$ and $3/4$, respectively.
Since $P_{F_4}\boxtimes  N_k=N_{3-k}$ for each $k\in\mathbb{Z}_4$, 
it follows that
$
\mathrm{Mod}=\{L(N_k;-1/20),L(N_k;3/4)\,|\,k\in\mathbb{Z}_4,k\mbox{ is odd}\}$ and
$\mathrm{Mod}_{\mathrm{tw}}=\{L(N_k;-1/20),L(N_k;3/4)\,|\,k\in\mathbb{Z}_4,k\mbox{ is even}\}.
$
The space $S$  has a~basis
$\chi_{L(N_0;-1/20)}$, $\chi_{L(N_2;3/4)}$, $\chi_{L(N_2;-1/20)}$, $\chi_{L(N_0;3/4)}$ with the conformal weights $-1/20,7/20,11/20,3/4$, respectively.
Since the central charge of $W(F_4)$ is $18/5$, the exponents of the basis of $S$ are $-1/5, 1/5, 2/5, 3/5$, which shows that
each element of $S$ satisfies~$(\flat_{18/5})$ by Lemma~\ref{sec:lemindex}.

\medskip\noindent{{\bf 5. $\mathfrak{g}=E_6$.}}
Since $U_{E_6}=V_1(A_5)\cong V_{A_5}$ and $P_{E_6}=V_1(A_5;\varpi_3)\cong V_{A_5+3\varpi_1}$,
the complete set of the irreducible modules over $V_{A_5}$ is 
given by
$
N_k:=V_{A_5+k\varpi_1}$ with $0\leq k\leq 5.
$
The conformal weights of $N_i$ $(0\leq i\leq 5)$ are $0, 5/12, 2/3, 3/4, 2/3$ and  $5/12$, respectively.
The space $T$ has a~basis
$(\chi_{N_i})_{0\leq i\leq 3}$.
Since $P_{E_6}\boxtimes  N_k=N_{k+3}$ for each $k\in\mathbb{Z}_6$, 
it follows that
$
\mathrm{Mod}=\{L(N_k;-1/20),L(N_k;3/4)\,|\,k\in\mathbb{Z}_6,k\mbox{ is odd}\}$ and
$\mathrm{Mod}_{\mathrm{tw}}=\{L(N_k;-1/20),L(N_k;3/4)\,|\,k\in\mathbb{Z}_6,k\mbox{ is even}\}.
$
The space $S$  has a~basis
$\chi_{L(N_0;-1/20)}$, $\chi_{L(N_4;3/4)}$, $\chi_{L(N_2;-1/20)}$, 
$\chi_{L(N_0;3/4)}$ with the conformal weights $-1/20,5/12,37/60,3/4$, respectively.
Since the central charge of $W(E_6)$ is $22/5$, the exponents of the basis of $S$ are $-7/30, 7/30, 13/30, 17/30$, which shows that
each element of $S$ satisfies~$(\flat_{22/5})$ by Lemma~\ref{sec:lemindex}.

\medskip\noindent{{\bf 6. $\mathfrak{g}=E_7$.}}
Since $U_{E_7}=V_1(D_6)\cong V_{D_6}$ and $P_{E_7}=V_1(D_6;\varpi_6)\cong V_{D_6+\varpi_6}$,
the complete set of the irreducible modules over $V_{D_6}$ up to the isomorphisms is given by
$N_0:=V_{D_6},N_1:=V_{D_6+k\varpi_1},N_2:=V_{D_6+\varpi_5}$
and $N_3:=V_{D_6+\varpi_6}$.
The conformal weights of $N_k$ $(0\leq k\leq 3)$ are $0,1/2,3/4$ and $3/4$, respectively.
The space $T$ has a~basis
$(\chi_{N_k})_{0\leq k\leq 2}$.
Since $P_{E_7}\boxtimes  N_k=N_{3-k}$ for each $k\in\mathbb{Z}_4$, 
it follows that
$
\mathrm{Mod}=\{L(N_k;-1/20),L(N_k;3/4)\,|\,k\in\mathbb{Z}_4,k\mbox{ is odd}\}$ and
$\mathrm{Mod}_{\mathrm{tw}}=\{L(N_k;-1/20),L(N_k;3/4)\,|\,k\in\mathbb{Z}_4,k\mbox{ is even}\}.
$
The space $S$  has a~basis
$\chi_{L(N_0;-1/20)}$, $\chi_{L(N_2;3/4)}$, $\chi_{L(N_2;-1/20)}$, 
$\chi_{L(N_0;3/4)}$ with the conformal weights $(-1/20,1/2,7/10,3/4)$, respectively.
Since the central charge of $W(E_7)$ is $27/5$, the exponents of the basis of $S$ are $-11/40, 11/40, 19/40, 21/40$, which shows that
each element of $S$ satisfies~$(\flat_{27/5})$ by Lemma~\ref{sec:lemindex}.

\medskip\noindent{{\bf 7. $\mathfrak{g}=E_8$.}}
Since $U_{E_8}=V_1(E_7)\cong V_{E_7}$ and $P_{E_8}=V_1(E_7;\varpi_7)\cong V_{E_7+\varpi_7}$,
the complete set of the irreducible modules over $V_{E_7}$ is 
$
N_k:=V_{E_7+k\varpi_7}$ with $k\in\mathbb{Z}_2.
$
The conformal weights of $N_0$ and $N_1$ are $0$ and $3/4$, respectively.
Since $P_{E_8}\boxtimes  N_k=N_{1+k}$ for each $k\in\mathbb{Z}_2$, 
it follows that
$
\mathrm{Mod}=\{L(N_k;-1/20),L(N_k;3/4)\,|\,k\in\mathbb{Z}_2,k\mbox{ is odd}\}$ and
$\mathrm{Mod}_{\mathrm{tw}}=\{L(N_k;-1/20),L(N_k;3/4)\,|\,k\in\mathbb{Z}_2,k\mbox{ is even}\}.
$
The space $S$ has a~basis
$\chi_{L(N_0;-1/20)}$, $\chi_{L(N_0;3/4)}$ with the conformal weights $-1/20,3/4$, respectively.
Since the central charge of $W(E_8)$ is $32/5$, the exponents of the basis of $S$ are $-19/60, 29/60$, which shows that
each element of $S$ satisfies $(\sharp_{\mu(19/5)})$ by Lemma~\ref{sec:lemindex}.
Since the left-hand side of~$(\flat_{32/5})$ is rewritten as a~linear combination of derivatives of the left-hand side of $(\sharp_{\mu(19/5)})$ as in~\eqref{eqn:rewritten} in Appendix~\ref{sec:basisrational}, each element of $S$ satisfies~$(\flat_{32/5})$.

\medskip\noindent
These proves the statements.
\end{proof}

\begin{remark}\label{e712}
(a) If we formally set $h^\vee=24$ (the number $24$ is not the dual Coxeter number of any Lie algebra
in the Deligne exceptional series) into~\eqref{eqn:deligne1},
we see that $(\flat_6)$ also has a~solution of CFT type.
The number $h^\vee=24$ appears in many studies related with the Deligne exceptional series (see e.g.~\cite{CdM}, \cite{MMS}, \cite{LM} and \cite{Kaw1}).
A similar phenomenon is observed in the study of the second order MLDEs (see \cite{MMS} and \cite{Kaw1}).
We now explain this phenomenon.
Let $\g$ be a Lie algebra in the Deligne exceptional series.
The characters of the affine VOA $V_1(\g)$ associated with 
$\g$ satisfy
$(\sharp_{\mu(c_{\g}/2)})$, where $c_\g$ is the central charge $c_\g=\dim\g/(1+h^\vee)$ of $V_1(\g)$. 
Since the Lie algebra $\g$ in the Deligne exceptional series satisfies $\dim\g=2(5h^\vee-6)(h^\vee+1)/(h^\vee+6)$ by the {\it Deligne dimension
formula} (see e.g.~\cite{CdM} and \cite{D}), the number $c_\g$ is a rational
function in $h^\vee$.
If we formally substitute $h^\vee=24$ into the rational function $c_\g$ in $h^\vee$,
we have $c_\g=38/5$, and $(\sharp_{\mu(19/5)})$
has solutions of CFT type and character type.
A vertex algebra $V_{E_{7+1/2}}$ associated with
the intermediate Lie algebra $E_{7+1/2}$ (\cite{LM,W}) was constructed in
\cite{Kaw1}, and modular invariant characters of $V_{E_{7+1/2}}$ were shown to satisfy $(\sharp_{\mu(19/5)})$ (see~\cite{Kaw1}).
If we formally substitute $h^\vee=3/2$ into the rational functions $c_\g$ and~\eqref{eqn:deligne1},
we have $c_\g=2/5$ and $s=-6/5$.
The MLDE~$(\sharp_{\mu(1/5)})$ has a~solution of CFT type with
the exponent $-c_\g/24(=-1/60)$, and~$(\flat_{-6/5})$ has a~solution $f=1$.


\medskip
\noindent (b) Let $\mathfrak{g}$ be a~Lie algebra in the Deligne exceptional series
and $M$ a~unique Ramond-twisted irreducible $\mathcal{W}_{-h^\vee/6}(\mathfrak{g},f_\theta)$-module with the minimum conformal weight.
Since the character of $M$ is of CFT type, we can obtain for any given $n\geq 0$ a~formula for $\dim\,M_n$
in terms of the variable $h^\vee$ by using~\eqref{eqn:recursion1} with $a_0=1$ (cf.~\cite{T}).
\end{remark}

\section{Characterization of $\mathcal{W}$-algebras associated with the Deligne exceptional series}\label{sec:extension}

Recall the {\bf Main Conditions} given in Introduction.
In this section we prove the following theorem.

\begin{theorem}\label{sec:identification2}
Let $W$ be a~simple, $C_2$-cofinite, $\Z_2$-rational, $\frac{1}{2}\mathbb{Z}_{\geq 0}$-graded, conical vertex operator algebra.
Suppose that $W$ satisfies the Main Conditions.
Then $W$ is isomorphic to $\mathcal{W}_{-h^\vee/6}(\mathfrak{g},f_\theta)$ 
associated with a~Lie algebra $\mathfrak{g}\neq A_1$ in the Deligne exceptional series.
\end{theorem}

In order to prove Theorem~\ref{sec:identification2}, we will show several propositions.
Let $W=\bigoplus_{n\in(1/2)\mathbb{Z}}W_n$ be a~simple, $C_2$-cofinite, 
$\Z_2$-rational, $\frac{1}{2}\mathbb{Z}_{\geq 0}$-graded, conical VOA with central charge $c\in\mathbb{C}$.
Assume that $W$ satisfies the Main Conditions.
Then we have a~complex number $s$ which satisfies~(a) in the Main Conditions.
Moreover, we have a~simple, $C_2$-cofinite, rational VOA
$V$ of CFT type,
and a~simple current $V$-module $P$ satisfying~(b).

We first give candidates of the number $s$.

\begin{proposition}\label{sec:prop1}
The number $s$ is one of $-3/5,2/5,6/5,12/5,18/5,22/5,27/5,6$ and $32/5$.
\end{proposition}

\begin{proof}
We first show that $s$ belongs to~\eqref{eqn:cand6} and that $s\neq -48/5$ nor $-38/5$.
By using (b),  we have an~irreducible Ramond-twisted $W$-module $L(V;-1/20)$
(this notation is defined in Definition~\ref{sec:def1}).
Since $c+3/5$ is the central charge of $V$, we have $c\in\Q$ (\cite{AM}).
Since $\chi_{L(V;-1/20)}(\tau)=q^{-c/24-1/20}(1+O(q))$ by (c),
the character of $L(V;-1/20)$ is of CFT type.
Since $W$ is $\Z_2$-rational, the space of characters of Ramond-twisted $W$-modules is $SL_2(\Z)$-invariant (\cite{V}), which shows that
 $\chi_{L(V;-1/20)}$ is not a~quasimodular form
of positive depth.
It then follows from Theorem~\ref{sec:modularcases} that $s$ belongs to~\eqref{eqn:cand6}.
The above $SL_2(\Z)$-invariance of characters also shows that $s\neq -48/5$ nor $-38/5$
by Proposition~\ref{sec:modularinv2}.

We finally show $s\neq -6/5,54/5,18,-66/5,-6$ nor $-8/5$.
Since the indices of $(\flat_s)$ do not have non-zero integral differences by (a),
we have $s\neq 54/5,18,-66/5$ nor $-6$.
By using (c) again, we have $\chi_{L(V;-1/20)}\geq \chi_{L(-3/5,-1/20)}$.
Here we say that two functions $f$ and $g$ of character type satisfy $f\geq g$ if $a_n\geq b_n$ for any
$n\in\Z_{\geq 0}$, where $f(\tau)=\sum_{n=0}^\infty a_n q^{n+\alpha}$ and
$g(\tau)=\sum_{n=0}^\infty b_n q^{n+\beta}$ denote the Fourier expansions of $f$
and $g$ with $a_0,b_0\neq 0$.
However, the solutions of CFT type of $(\flat_{-6/5})$ and $(\flat_{-8/5})$,
which can be found in Appendix~\ref{sec:basisrational} {\bf (c)} and {\bf (q)},
are not greater than nor equal to $\chi_{L(-3/5,-1/20)}$ because
 $\chi_{L(-3/5,-1/20)}(\tau)=q^{1/40}(1+q+q^2+2q^3+O(q))$.
Thus $s\neq -6/5$ nor $-8/5$, which completes the proof.
\end{proof}

We next construct a~VOA $R$ from $V$ and $P$.
We will later show that the characters of $R$ satisfy 
a~specific MLDE of order 2
if $s\neq 32/5$. 
(Then by Theorem~\ref{sec:identification1}, we specify $R$, which enables us to determine $V$ and $P$.)
Let $A_1$ be a~simple Lie algebra of type $A_1$ and $\alpha$ a~positive root of $A_1$.
The affine VOA $V_1(A_1)$ has a~unique non-vacuum irreducible module $V_1(A_1;\alpha/2)$.
It is well known that $V_1(A_1;\alpha/2)$ is a~simple current with $V_1(A_1;\alpha/2)\boxtimes V_1(A_1;\alpha/2)=V_1(A_1)$ and the conformal weight of $V_1(A_1;\alpha/2)$ is $1/4$.

\begin{proposition}\label{sec:scer}
There is a~unique vertex operator algebra $R$ of the form
 $R\cong V\otimes V_1(A_1)\oplus P\otimes V_1(A_1;\alpha/2)$,
 which is called the {\it simple current extension} of $V\otimes V_1(A_1)$ by $P\otimes V_1(A_1;\alpha/2)$.
\end{proposition}

\begin{proof}
Put $A=V\otimes V_1(A_1)$ and $B=P\otimes V_1(A_1;\alpha/2)$.
Since the conformal weight of $B$ is an~integer,
we have a~super\,VOA $R$ such that $R=A\oplus B$,
which is unique up to isomorphisms of super\,VOAs (see~\cite{C}).
Let $u$ be a~non-zero element of $P$ and $v$ a~highest weight vector
of $L(-3/5,3/4)$.
Since $W$ is simple, $w=u\otimes v\in W$ satisfies $Y(w,z)w\neq 0$.
Take $a\in\mathbb{Z}$ such that $z^a Y(w,z)w\in W[[z]]$ and $z^aY(w,z)w|_{z=0}\neq 0$.
As $W$ is a~VOA, it follows from Lemma~\ref{sec:lemsuper} that $e^{-\pi \sqrt{-1}a}=1$.
Since $v$ satisfies $v(1/2)v=k|0\rangle$ with non-zero $k\in\mathbb{C}$, 
where $|0\rangle$ is the vacuum vector of $L(-3/5,0)$, it follows that $z^{3/2} Y(v,z)v\in L(-3/5,0)[[z]]$ and $z^{3/2}Y(v,z)v|_{z=0}\neq 0$.
Therefore, the element $u$ satisfies $z^{a-3/2} Y(u,z)u\in V[[z]]$ and $z^{a-3/2}Y(u,z)u|_{z=0}\neq 0$.
A~highest weight vector $e^{\alpha/2}\in V_1(A_1;\alpha/2)$ of $\mathfrak{h}$-weight $\alpha/2$ satisfies
$e^{\alpha/2}(-3/2)e^{\alpha/2}=k_1e^\alpha$ with a~non-zero number $k_1\in\mathbb{C}$, 
where $e^\alpha$ is a~homogeneous element of $\mathfrak{h}$-weight $\alpha$.
It then follows that $z^{-1/2} Y(e^{\alpha/2},z)e^{\alpha/2}\in V_1(A_1)[[z]]$ and $z^{-1/2}Y(e^{\alpha/2},z)e^{\alpha/2}|_{z=0}\neq 0$.
Therefore, $u_1=u\otimes e^{\alpha/2}\in R$ satisfies
$z^{a-2} Y(u_1,z)u_1\in R[[z]]$ and $z^{a-2}Y(u_1,z)u_1|_{z=0}\neq 0$.
It follows from Lemma~\ref{sec:lemsuper} that
$R$ is a~VOA since $e^{-\pi\sqrt{-1}(a-2)}=1$.
\end{proof}

We next describe the irreducible modules over $V$. 
It turns out in the following proposition that the space of characters of $V$-modules is
 at most 4-dimensional, which will show that the space of characters of $R$-modules is at most $2$-dimensional.
We denote by $h_M$ the conformal weight of an~irreducible $V$-module $M$.
Let $\cA$ denote the set of isomorphism classes of irreducible $V$-modules
$M$ which satisfy $h_M\not\in\Z$ and $h_{M\boxtimes P}-h_M\in 3/4+\Z$.

\begin{proposition}\label{sec:vmodules}
The set $\{V,P\}\sqcup \cA\sqcup (\cA\boxtimes P)$
is the complete set of the isomorphism classes of irreducible $V$-modules,
where $\cA\boxtimes P=\{M\boxtimes P|M\in\cA\}$.
Moreover, for any $M,N\in\cA$, both equalities $\chi_M=\chi_N$ and $\chi_{M\boxtimes P}
=\chi_{N\boxtimes P}$ hold.
\end{proposition}

To show the proposition, we will prove several lemmas.

%
%
We write the conformal weight of each irreducible Ramond-twisted $W$-module $L$ as $r_L$.

\begin{lemma}\label{sec:lemcw}
{\rm(I)} For any irreducible $V$-module $M$, the conformal weights of $M$ and $M\boxtimes P$ satisfy either
$h_{M\boxtimes P}-h_{M}\in3/4+\mathbb{Z}$ or $1/4+\mathbb{Z}$.
{\rm(II)}
Let $M$ be an~irreducible $V$-module such that $h_{M\boxtimes P}-h_{M}\in3/4+\mathbb{Z}$. Then the following assertions {\rm (1)--(6)} hold:
{\rm(1)}~If $h_{M\boxtimes P}-h_M\geq -1/4$, then $r_{L(M;-1/20)}=h_M-1/20$.
{\rm(2)}~If $h_{M\boxtimes P}-h_M\leq -1/4$, then $r_{L(M;-1/20)}=h_{M\boxtimes P}+1/5$.
{\rm(3)}~If $h_{M\boxtimes P}-h_M\geq 3/4$, then $r_{L(M;3/4)}=h_M+3/4$.
{\rm(4)}~If $h_{M\boxtimes P}-h_M\leq 3/4$, then $r_{L(M;3/4)}=h_{M\boxtimes P}$.
{\rm(5)}~$r_{L(M;3/4)}-r_{L(M;-1/20)}\in 4/5+\Z$.
{\rm(6)}~$r_{L(V;-1/20)}=-1/20$ and $r_{L(V;3/4)}=3/4$.
\end{lemma}

\begin{proof}
The assertion (I) follows from condition (h).
Since $r_{L(M;-1/20)}=\min\{h_M-1/20,h_{M\boxtimes P}+1/5\}$
and $r_{L(M;3/4)}=\min\{h_M+3/4,h_{M\boxtimes P}\}$,
we have assertions (1)--(5).
The assertion (6) follows from (e).
\end{proof}

\begin{lemma}\label{sec:lemvac}
Let $M$ be an~irreducible $V$-module such that $h_{M\boxtimes P}-h_M\in3/4+\mathbb{Z}$.
If $h_M\in\mathbb{Z}$, then $M$ is isomorphic to $V$.
\end{lemma}

\begin{proof}
If $h_M=0$, then $M\cong V$ by (d).
We show $h_M=0$ by contradiction.
Suppose on the contrary that $h_M\in \Z_{>0}$ and
let $k$ be an~element of $\{-1/20,3/4\}$.
Since both $L(M;k)$ and $L(V;k)$ are Ramond-twisted by Lemma~\ref{sec:3414}, the number $r_{L(M;k)}-r_{L(V;k)}$ is an~integer.

We now show $r_{L(M;k)}=r_{L(V;k)}$ by contradiction.
Suppose that $r_{L(M;k)}\neq r_{L(V;k)}$, which implies that $\chi_{L(M;k)}$ and $\chi_{L(V;k)}$ are linearly-independent.
Then we have $i,j\in\{1,\ldots,4\}$ such that
$i\neq j$, $\alpha_i-\alpha_j\in\Z$ and $r_{L(V;k)}\in\alpha_i+\Z$.
Here $\alpha_1,\ldots,\alpha_4$ are the indices of $(\flat_s)$.
It follows by (a) that $\alpha_i=\alpha_j$.
Since $(\flat_s)$ has at most one pair of multiple indicial roots,
we have a~solution with a~logarithmic term, and the space
spanned by the solutions of $(\flat_s)$ of character type whose exponents belong to 
$\alpha_i+\Z$ is one-dimensional.
It contradicts the linearly-independence of $\chi_{L(M;k)}$ and $\chi_{L(V;k)}$.
Hence, $r_{L(M;k)}= r_{L(V;k)}$.

The equality $r_{L(M;-1/20)}=r_{L(V;-1/20)}$, Lemma~\ref{sec:lemcw} and $h_M\neq 0$ show
 that $h_{M\boxtimes P}+1/5=-1/20$ and $h_{M\boxtimes P}-h_M<-1/4$.
Therefore, $h_{M\boxtimes P}=-1/4$ and $h_M>0$.
By using Lemma~\ref{sec:lemcw} again, we have $r_{L(M;3/4)}=-1/4$.
As $r_{L(V;3/4)}=\min\{3/4,h_P\}$, condition (c) leads to $r_{L(V;3/4)}\geq 0$, 
which contradicts $r_{L(M;3/4)}=r_{L(V;3/4)}$.
Thus we have $h_M=0$, which completes the proof.
\end{proof}

\begin{lemma}\label{sec:lemcharv}
Let $M$ and $N$ be elements of $\cA$ and suppose that {\rm(A)} $r_{L(M;-1/20)}=r_{L(N;-1/20)}$ and 
{\rm(B)} $r_{L(M;3/4)}=r_{L(N;3/4)}$ hold.
Then $\chi_{M}=\chi_{N}$ and $\chi_{M\boxtimes P}=\chi_{N\boxtimes P}$.
\end{lemma}

\begin{proof}
We first show that either
\begin{equation}\label{eqn:lemcharv1}
h_{M}=h_{N} \quad\mbox{or}\quad h_{M\boxtimes P}=h_{N\boxtimes P}
\end{equation}
holds,
whose proof is divided into 4 cases: {\bf1.}~$h_{M\boxtimes P}-h_{M}<-1/4$, 
{\bf2.}~$h_{M\boxtimes P}-h_{M}=-1/4$, 
{\bf3.}~$h_{M\boxtimes P}-h_{M}=3/4$, and
{\bf4.}~$h_{M\boxtimes P}-h_{M}> 3/4$.

\medskip\noindent
{\bf1.}~By Lemma~\ref{sec:lemcw}, we have $r_{L(M;-1/20)}=h_{M\boxtimes P}+1/5$
and $r_{L(M;3/4)}=h_{M\boxtimes P}$.
Suppose that $h_{N\boxtimes P}-h_N\leq 3/4$, which implies $r_{L(N;3/4)}=h_{N\boxtimes P}$.
Then we have $h_{M\boxtimes P}=h_{N\boxtimes P}$ by assumption~(2).
Suppose that $h_{N\boxtimes P}-h_N>3/4$.
It then follows that $r_{L(N;-1/20)}=h_{N}-1/20$
and $r_{L(N;3/4)}=h_{N}+3/4$,
which contradict (A) or (B).
Thus, we have~\eqref{eqn:lemcharv1} in this case.

\medskip\noindent
{\bf2.} Suppose $h_{M\boxtimes P}-h_{M}=-1/4$.
Then $r_{L(M;-1/20)}=h_{M}-1/20=h_{M\boxtimes P}+1/5$
by Lemma~\ref{sec:lemcw}.
Since $r_{L(N;-1/20)}=\min\{h_{N}-1/20,h_{N\boxtimes P}+1/5\}$,
the assumption (A) forces~\eqref{eqn:lemcharv1}.

\medskip\noindent
The assertions {\bf3}--{\bf4} are proved in a~similar way, which completes the proof
of~\eqref{eqn:lemcharv1}.

We now prove the lemma. Suppose first that $h_M=h_N$. Then we have $\chi_M=\chi_N$ by (e).
It then follows from (A) and (f) that $\chi_{L(M;-1/20)}=\chi_{L(N;-1/20)}$, which shows $\chi_{M\boxtimes P}=\chi_{N\boxtimes P}$.
In a~similar way, we have $\chi_M=\chi_N$ 
and $\chi_{M\boxtimes P}=\chi_{N\boxtimes P}$ when $h_{M\boxtimes P}=h_{N\boxtimes P}$.
Thus we have the lemma.
\end{proof}

Set $C_s=\{(i,j)\,|\,1\leq i,j\leq 4, \alpha_j-\alpha_i\in 4/5+\mathbb{Z}\}$.
We have completed the preliminaries on the proof of Proposition~\ref{sec:vmodules}.

\begin{proof}[Proof of Proposition~\ref{sec:vmodules}]
Let $M$ and $N$ be elements of $\cA$.
Since $P\boxtimes P\cong V$, by Lemma~\ref{sec:lemcw} (I) and Lemmas~\ref{sec:lemvac}--\ref{sec:lemcharv}, it suffices to show $r_{L(M;k)}=r_{L(N;k)}$ for each $k=-1/20,3/4$. 
We devide the proof into 3 parts: {\bf1.}~$s\neq 6$ nor $6/5$, {\bf2.}~$s=6$, 
and {\bf3.}~$s=6/5$.

\medskip\noindent
{\bf1.}~
Suppose that $s\neq 6$ nor $6/5$.
It then follows by Proposition~\ref{sec:prop1} that $C_s=\{(1,2),(3,4)\}$.
Since $h_M,h_N\not\in\Z$, Lemma~\ref{sec:lemcw} (5) shows that the following (A) or (B) holds:

\noindent
(A) $r_{L(V;-1/20)}=\alpha_1$, $r_{L(V;3/4)}=\alpha_2$,
 $r_{L(M;-1/20)}=r_{L(N;-1/20)}=\alpha_3$, and $r_{L(M;3/4)}=r_{L(N;3/4)}=\alpha_4$;

\noindent
  (B) $r_{L(V;-1/20)}=\alpha_3$, $r_{L(V;3/4)}=\alpha_4$,
 $r_{L(M;-1/20)}=r_{L(N;-1/20)}=\alpha_1$, and $r_{L(M;3/4)}=r_{L(N;3/4)}=\alpha_2$.

\noindent
Hence we have $r_{L(M;k)}=r_{L(N;k)}$ for $k=-1/20,3/4$.

\medskip\noindent
{\bf2.}~
If $s=6$, then there exists a~unique pair of $1\leq i<j\leq 4$
such that $\alpha_i=\alpha_j$.
Therefore, there is a~solution of~$(\flat_s)$ with a~logarithmic term, and the space of characters of Ramond-twisted $W$-modules is at most $3$-dimensional.
Hence, there are at most 3 conformal weights of $W$-modules.
Since $h_M,h_N\not\in\Z$, the numbers $r_{L(M;-1/20)}$ and $r_{L(N;-1/20)}$
do not belong to $r_{L(V;-1/20)}+\Z$.
Hence, $r_{L(M;-1/20)}=r_{L(N;-1/20)}=r_{L(V;3/4)}$,
and eventually, $r_{L(M;3/4)}=r_{L(N;3/4)}$.

\medskip\noindent
{\bf3.}~
If $s=6/5$, then we see that 
$(\alpha_i)_{1\leq i\leq 4}=(-1/10,7/10,3/10,1/10)$ and $C_{6/5}=\{(3,4),(4,1),(1,2)\}$.
Since there are no integral differences between the indices, the exponent
of any irreducible Ramond-twisted $W$-module belongs to
$\{\alpha_i|1\leq i\leq 4\}$.
Let $N^0,N^1$ and $N^2$ be irreducible $V$-modules such that
$h_{N^i\boxtimes P}-h_{N^i}\in3/4+\mathbb{Z}$ for each $0\leq i\leq 2$.
It then follows that $r_{L(N^i;-1/20)}\in\{\alpha_1,\alpha_3,\alpha_4\}$.
Hence, it suffices to show that $\#\{r_{L(N^i;-1/20)}|i=0,1,2\}\leq 2$.
Suppose on the contrary that 
the exponents of $L(N^0;-1/20),L(N^1;-1/20)$ and $L(N^2;-1/20)$ are 
$\alpha_3,\alpha_4$ and $\alpha_1$, respectively.
It then follows that the exponent of  $L(N^1;3/4)$ is $\alpha_1$.
Therefore, (g) implies that the $W^{(0)}$-module $N^2\otimes L(-3/5,-1/20)$ is isomorphic to one of
the modules $N^1\otimes L(-3/5,3/4)$, $(N^1\boxtimes P)\otimes L(-3/5,0)$,
$(N^1)'\otimes L(-3/5,3/4)'$ and $(N^1\boxtimes P)'\otimes L(-3/5,0)'$.
It contradicts $L(-3/5,0)'\cong L(-3/5,0)$ and $L(-3/5,3/4)'\cong L(-3/5,3/4)$.
Thus,  $\#\{r_{L(N^i;-1/20)}|i=0,1,2\}\leq 2$, which shows $r_{L(M;k)}=r_{L(N;k)}$
for $k=-1/20$ and $3/4$.

The discussions {\bf1}--{\bf3} show the lemma.
\end{proof}

We now give a~proof of Theorem~\ref{sec:identification2}.

\begin{proof}[{\bf Proof of Theorem~\ref{sec:identification2}}]
We divide a~proof into two parts (A): $\cA=\emptyset$ and (B): $\cA\neq \emptyset$.

\medskip\noindent
{\bf The case (A)}
Suppose that $\cA=\emptyset$.
It then follows from Proposition~\ref{sec:vmodules} and~(h) that the space of characters of Ramond-twisted  $W$-modules
has a~basis $(\chi_{L(V;-1/20)},\chi_{L(V;3/4)})$.
Then we see that $s=32/5$ by Propositions~\ref{sec:modularinv1} and~\ref{sec:prop1}. 
Since any solution of $(\flat_{32/5})$ of CFT type has an~exponent $-19/60$ by \S~\ref{sec:cft},
we see that $\chi_{L(V;-1/20)}$ has the exponent $-19/60$ and $\chi_{L(V;3/4)}$ is not of CFT type.
Therefore, the central charge of $V$ is 7 and $h_P=3/4$.
Then Lemma~\ref{sec:lemindex} shows that the characters of $V$ satisfy~$(\sharp_{\mu(7/2)})$, which implies $\dim V_1=133$.
By using (d), we see that the VOA $V$ is self-dual.
Therefore, by Theorem~\ref{sec:propsecond} we have
 $V\cong V_1(E_7)$ and $P\cong V_1(E_7;\varpi_7)$,
and then $V\cong U_{E_8}$ and $P\cong P_{E_8}$.
By Lemma~\ref{sec:branching}, we see that
$W\cong \mathcal{W}_{-5}(E_8,f_\theta)$.

\medskip\noindent
{\bf The case (B)}
Suppose that $\cA\neq\emptyset$.
Then we have an~element $M$ of $\cA$.
Set $L^0=L(V;-1/20)$, 
$L^1=
L(V;3/4)$,
$L^2=L(M;-1/20)$ and
$L^3=L(M;3/4)$ with the characters $\chi_0,\chi_1,\chi_2$ and $\chi_3$, respectively.
It then follows from Proposition~\ref{sec:vmodules} that the space of characters of Ramond-twisted  $W$-modules
has a~spanning set $\{\chi_0,\chi_1,\chi_2,\chi_3\}$.
Define $m\in\Z$ by $h_{M\boxtimes P}-h_M=3/4+m$. 
By Lemma~\ref{sec:lemcw}, the conformal weights $r_i$ of $L^i$ are
\[
(r_0,r_1,r_2,r_3)=\begin{cases}(-1/20,3/4,h_M-1/20,h_M+3/4)&
\mbox{if}\ m\geq0,\\
(-1/20,3/4,h_M+m+19/20,h_M+m+3/4)&
\mbox{if}\ m\leq -1.
\end{cases}
\]
By comparing the exponents of $L^0,L^1,L^2$ and $L^3$ and indicial roots of~\eqref{eqn:fourth}, we see that 
$m\leq-1$, $s=c$ and $c=12h_M+12m+42/5$.

By Proposition~\ref{sec:scer}, we have the simple current extension $R$
of $V\otimes V_1(A_1)$ by $P\otimes V_1(A_1;\alpha/2)$.
It follows from Proposition~\ref{sec:vmodules} that the space of characters of $R$-modules has a~basis which consists of the characters of 
$R$ itself
and $K=M\otimes V_1(A_1)\oplus (M\boxtimes P)\otimes V_1(A_1;\alpha/2)$.
Since $m\leq -1$, the conformal weight of $K$ is $h_M+m+1$.
As the central charge of $V$ is $12h_M+12m+9$, 
 that of $R$ is $12h_M+12m+10$.
Since the exponents of the characters of $R$ coincide with the indices of~$(\sharp_{\mu(6h_M+6m+5)})$, the characters of $R$ satisfy 
$(\sharp_{\mu(6h_M+6m+5)})$ by Lemma~\ref{sec:lemindex}.
It follows from (d) that $R$ is self-dual.
As $A_1\subset R_1$, Theorem~\ref{sec:propsecond} implies that $R\cong V_1(\mathfrak{g})$
with a~Lie algebra $\mathfrak{g}$ in the Deligne exceptional series.
Since $P\neq 0$ and $R$ has two irreducible modules, it follows that $\mathfrak{g}\neq A_1$ and $E_8$.
Therefore, we have $V=U_\mathfrak{g}$, $P=P_\mathfrak{g}$ and that
 $W$ is isomorphic to $\mathcal{W}_{-h^\vee/6}(\mathfrak{g},f_\theta)$
by Lemma~\ref{sec:branching}.
Thus we have proved Theorem~\ref{sec:identification2}.
\end{proof}

\section{The other identification problems}\label{sec:complements}

In Theorem~\ref{sec:identification1}, we find rational VOAs whose characters satisfy~\eqref{eqn:fourth} when $s$ has the form~\eqref{eqn:deligne1} with
the dual Coxeter number $h^\vee$ of one of the Lie algebras in the Deligne exceptional series. 
Moreover, we discuss the case of $s=6$ and $s=-6/5$ (these numbers are given by \eqref{eqn:deligne1} with $h^\vee=24$ and $3/2$) in Remark~\ref{e712}.
In this section we discuss about the problem how to find rational VOAs whose characters satisfy~\eqref{eqn:fourth}, where $s$ is one of the remaining numbers appearing in~\eqref{eqn:cand6}.

\medskip\noindent
{\bf (1) $\mathbf{s=-48/5}$.} By Proposition~\ref{sec:modularinv2}, the characters of a~simple $C_2$-cofinite rational VOA of CFT type 
do not satisfy~$(\flat_{-48/5})$.
The affine VOA $V=V_{-3/2}(G_2)$ at admissible level
$-3/2$ is not $C_2$-cofinite but semi-simple in the category $\mathcal{O}$.
The VOA $V$ has two ordinary irreducible modules with the conformal weights  $0$ and $4/5$.
The characters of ordinary modules over affine VOAs at admissible levels
satisfy MLDEs (\cite{AK}).
Since the central charge of $V_{-3/2}(G_2)$ is $-42/5$,
the irreducible characters have the forms
$q^{7/20}(1+14q+\sum_{n=2}^\infty a_n q^n)$ and $q^{23/20}(7+\sum_{n=1}^\infty b_nq^n)$ with non-negative integers $a_n$ and $b_n$.
Since the solutions $f_0$ and $7 f_{4/5}$ of~\eqref{eqn:fourth} defined in {\bf (a)} of
Appendix~\ref{sec:basisrational}
also have this form,
we propose a~conjecture below.
\begin{conjecture}
The characters of the irreducible ordinary modules over $V_{-3/2}(G_2)$
coincide with the solutions $f_0$ and $7 f_{4/5}$ of~$(\flat_{-48/5})$.
\end{conjecture}

Note that if this conjecture is true, The solution $f_0$ in {\bf(a)} in Appendix~\ref{sec:basisrational} gives a~character formula of $V_{-3/2}(G_2)$
in terms of modular forms in Table~\ref{tb:forms}.

\medskip\noindent
{\bf (2) $\mathbf{s=-38/5}$.} The same argument as in~{\bf(1)} suggests the following conjecture.

\begin{conjecture}
There exists an~extension $V$ of the admissible affine vertex operator algebra $V_{-4/3}(A_2)$
of the form $V\cong V_{-4/3}(A_2)\oplus V_{-4/3}(A_2;2\varpi_1)\oplus V_{-4/3}(A_2;2\varpi_2)$.
Moreover, the characters of the ordinary irreducible modules over $V$ coincide with
the solutions $f_{-8/15}$ and $f_{-1/3}$ of~$(\flat_{-38/5})$ up to
a~scalar multiple.
\end{conjecture}

\medskip\noindent
{\bf (3)} 
Let $s$ be an~element of $\{54/5,18,-66/5,-6\}$.
Let $V$ be a~rational VOA and suppose that the characters of $V$ satisfy~\eqref{eqn:fourth}.
There is a~solution of vacuum type of~\eqref{eqn:fourth} of exponent $-s/24$.
However, since $1-s/24$ is the exponent of a~solution of~\eqref{eqn:fourth},  
it is difficult to know the dimension of the Lie algebra $V_1$.
If we know $\dim V_1$, we have finite candidates of the structure of $V_1$,
and it often enables us to find affine sub\,VOA of $V$.
Then by using the representation theory of affine VOA's, there is a chance
to determine the structure of $V$.
However, since it is difficult to determine $\dim\,V_1$ in this case,
we can not use this discussion, which makes it difficult to find  $V$.
When $s=18$ or $-6$, the character of the trivial VOA $V=\mathbb{C}$ satisfies~\eqref{eqn:fourth} since $f=1$ is a~solution of~\eqref{eqn:fourth}.

\medskip\noindent
{\bf (4) $\mathbf{s=-8/5}$.}
The characters of the Virasoro minimal model $L(-22/5,0)$ coincide with
the solutions $f_0$ and $f_{-1/5}$ of~$(\flat_{-8/5})$.

\appendix

\section{Modular forms, differential relations and functional equations}\label{sec:modularforms}
In this appendix we introduce some modular forms of levels $N=2$--$5$ and $15$
 to describe solutions of~\eqref{eqn:fourth} in Appendices~\ref{sec:basisrational}--\ref{sec:quasimodular}.
 The modular forms are listed in Table~\ref{tb:forms}.
  Moreover, we list some differential relations and functional equations of modular forms in {\bf(a)}--{\bf(g)} below.
By using these, we can prove that the functions listed in Appendices~\ref{sec:basisrational}--\ref{sec:quasimodular} are solutions of~$(\flat_s)$.

The first column of Table~\ref{tb:forms} is the names of the modular forms $f$.
The second column shows the Fourier expansion $f(\tau)=\sum_{n=0}^\infty a_nq^{n+\alpha}$ of $f$, where $\tau\in\mathbb{H}$ and $q=e^{2\pi i\tau}$.
The third and fourth columns are the weight and level of $f$, respectively.

\begin{table}[bht]
\caption{Modular Forms}\label{tb:forms}
\begin{tabular}{|c|l|c|c|}
\hline
Name& Fourier Expansion & Weight & Level \\ \hline
$H_2$& $1+24\sum_{n=1}^{\infty}\left( \sum_{\substack{d | n,\ d\text{:odd}}}d\right) q^{n}$ & 2 & 2\\
$\Delta_2$ & $\eta(q^2)^{16}/\eta(q)^8$ & 2 & 2\\
$    I_{3}$ &$1+6\sum_{n=1}^{\infty}\left( \sum_{d | n}\left(\frac{d}{3}\right)\right) q^{n}$
&$1$ & $3$\\
$    \Delta_3$ &$\eta(q^3)^3/\eta(q)$ &$1$&$3$\\
$    \theta$&$\sum_{n\in\mathbb{Z}}q^{n^2}$&$1/2$&$4$\\
$\Delta_4$&$\eta(q^4)^2/\eta(q^2)$&$1/2$&$4$\\
$\psi_1$ & $\eta(q)^{2/5}\Bigl\{q^{-1/60}\prod_{\substack{n>0 \\ n\not\equiv0, \pm2\bmod{5}}}(1-q^n)^{-1}\Bigr\}$
&$1/5$ & $5$\\
$\psi_2$&$\eta(q)^{2/5}\Bigl\{q^{11/60}\prod_{\substack{n>0 \\ n\not\equiv0, \pm1\bmod{5}}}(1-q^n)^{-1}\Bigr\}$
&$1/5$ & $5$\\
$    I_{15}$&$\eta(q^3)^2\eta(q^5)^2 / \eta(q)\eta(q^{15})$
&$1$&$15$\\
$    \Delta_{15}$&$\eta(q)^2\eta(q^{15})^2/\eta(q^3)\eta(q^5)$
&$1$&$15$\\ \hline
\end{tabular}
\end{table}

These modular forms and Eisenstein series satisfy the following differential relations and functional equations.

\medskip\noindent
{\bf (a) Level~2}
\begin{align*}
&6H_2'=E_2 H_2-H_2^2+192\Delta_2^2\,,
&
&E_4=H_2^2+192\Delta_2^2\,,\\
&6\Delta_2'=(E_2+2H_2)\Delta_2\,,
&
&E_6=(H_2^2-576\Delta_2^2)H_2\,.
\end{align*}

\medskip\noindent
{\bf (b) Level~3}
\begin{align*}
&12I_3'=E_2I_3-I_3^3+108\Delta_3^3\,,
&
&E_4=I_3(I_3^3+216\Delta_3^3)\,,\\
&12\Delta_3'=(E_2+3I_3^2)\Delta_3\,,
&
&E_6=I_3^6-540I_3^3\Delta_3^3-5832\Delta_3^6\,.
\end{align*}

\medskip\noindent
{\bf (c) Level~4}
\begin{align*}
&24\theta'=(E_2-\theta^4+80\Delta_4^4)\theta\,,
&
&E_4=\theta^8+224\theta^4\Delta_4^4+256\Delta_4^8\,,\\
&24\Delta_4'=(E_2+5\theta^4-16\Delta_4^4)\Delta_4\,,
&
&E_6=(\theta^4+16\Delta_4^4)(\theta^8-544\theta^4\Delta_4^4+256\Delta_4^8)\,.
\end{align*}

\medskip\noindent
{\bf (d) Level~5}
\begin{align*}
&60\psi_1'=(E_2-\psi_1^{10}+66\psi_1^5\psi_2^5+11\psi_2^{10})\psi_1\,,
\quad
60\psi_2'=(E_2+11\psi_1^{10}-66\psi_1^5\psi_2^5-\psi_2^{10})\psi_2\,,\\
&E_4=\psi_1^{20}+228\psi_1^{15}\psi_2^5+494\psi_1^{10}\psi_2^{10}-228\psi_1^5\psi_2^{15}+\psi_2^{20}\,,\\
&E_6=(\psi_1^{10}+\psi_2^{10})(\psi_1^{20}-522\psi_1^{15}\psi_2^5-10006\psi_1^{10}\psi_2^{10}+522\psi_1^5\psi_2^{15}+\psi_2^{20})\,.
\end{align*}

\medskip\noindent
{\bf (e) Level~10}
\begin{align*}
&6H_2(q^5)'=5\{E_2(q^5)H_2(q^5)-H_2(q^5)^2+192\Delta_2(q^5)^2\}\,,\\
&6\Delta_2(q^5)'=5\Delta_2(q^5)\{E_2(q^5)+H_2(q^5)\}\,,\\
&5E_2(q^5)=E_2(q)+4\{\psi_1(q)^{10}+\psi_2(q)^{10}\}\,,\\
&55H_2(q^5)=-3\{5\psi_2(q)^{10}+7\psi_1(q^2)^5\psi_1(q)^5-21\psi_1(q^2)^{5}\psi_2(q)^5+30\psi_2(q^2)^5\psi_1(q)^5\}\\
&\hspace*{30ex}+76\psi_1(q^2)^{10}-78\psi_1(q^2)^5\psi_2(q^2)^5+70\psi_2(q^2)^{10}\,,\\
&3000\Delta_2(q^5)^2=2\{H_2(q)^2-60\Delta_2(1)^2\}-H_2(q)\psi_1(q)^{10}\\
&\hspace*{15ex}-H_2(q^5)\{38\psi_1(q)^{10}+132\psi_1(q)^5\psi_2(q)^5+37\psi_2(q)^{10}\}-720\Delta_2(q)\Delta_2(q^5)\\
&\hspace*{20ex}+2H_2(q)\{8\psi_1(q^2)^{10}-26\psi_1(q^2)^5\psi_2(q^2)^{5}+3\psi_2(q^2)^{10}\}\\
&\hspace*{30ex}+10H_2(q^5)\{2\psi_1(q^2)^{10}-2\psi_1(q^2)\psi_2(q^2)^{5}+3\psi_2(q^2)^{10}\}
\end{align*}

\medskip\noindent
{\bf (f) Level~15}
\begin{align*}
&12I_{15}'=(E_2-5I_{15}^2-2I_{15}\Delta_{15}-13\Delta_{15}^2+I_3^2)I_{15}\,,\\
&12\Delta_{15}'=(E_2+13I_{15}^2-2I_{15}\Delta_{15}+5\Delta_{15}^2-2I_3^2)\Delta_{15}\,.\\
&12I_3(q^5)'=E_2(q)I_3(q^5)+I_3(q)\{4I_{15}(q)+4I_{15}(q)\Delta_{15}(q)+2\Delta_{15}(q)^2\}\\
&\hspace*{30ex}
-I_3(q^5)\{5I_{15}(q)^2+2I_{15}(q)\Delta_{15}(q)-5\Delta_{15}(q)^2\}\,,\\
&I_3(q)^2=6I_{15}(q)^2+6\Delta_{15}(q)^2-5I_{3}(q^5)^2\,,\\ 
&I_3(q)I_3(q^5)=I_{15}(q)^2+4I_{15}(q)\Delta_{15}(q)-\Delta_{15}(q)^2\,,\\
&I_3(q)^3=6I_3(q)\{I_{15}(q)^2+\Delta_{15}(q)^2\}-5I_{3}(q^5)\{I_{15}(q)^2+4I_{15}(q)\Delta_{15}(q)-\Delta_{15}(q)^2\}\,,\\
&108\Delta_3(q)^3=I_{3}(q)\{25I_{15}(q)^2-2I_{15}(q)\Delta_{15}(q)-\Delta_{15}(q)^2\}\\
&\hspace*{30ex}-5I_{3}(q^5)\{5I_{15}(q)^2+8I_{15}(q)\Delta_{15}(q)+\Delta_{15}(q)^2\}\,,\\
&120\psi_2(q)^{10}=45\{I_{15}(q)^2-5\Delta_{15}(q)^2\}-6I_{3}(q)\{14I_{15}(q)-2I_{15}(q)\Delta_{15}(q)-5\Delta_{15}(q)^2\}\\
&\hspace*{30ex}+3I_{3}(q^5)\{8I_{15}(q)+64\Delta_{15}(q)-5I_{3}(q)\}.
\end{align*}

\medskip\noindent
{\bf (g) Level~20}
\begin{align*}
&\theta(q^5)'=\theta(q^5)\{E_2(q)+4\psi_1(q)^{10}+4\psi_2(q)^{10}-5\theta(q^5)^4+400\Delta_4(q^5)^4\}\,,\\
&\Delta_4(q^5)'=\Delta_4(q^5)\{E_2(q)+4\psi_1(q)^{10}+4\psi_2(q)^{10}+25\theta(q^5)^4-80\Delta_4(q^5)^4\}\,,\\
&\theta(q)^5=80\theta(q)\Delta_4(q)^4-5\theta(q)\{\psi_1(q)^5+\psi_2(q)^5\}\\
&\hspace*{35ex}+6\theta(q^5)\{\psi_1(q)^{10}-11\psi_1(q)^5\psi_2(q)^5-\psi_2(q)^{10}\}\,,\\
&\Delta_{4}(q)^5=5\Delta_4(q)\{\theta(q)^4-\psi_1(q)^{10}-\psi_2(q)^{10}\}\\
&\hspace*{35ex}+6\Delta_4(q^5)\{\psi_1(q)^{10}-11\psi_1(q)^{5}\psi_2(q)^5-\psi_2(q)^{10}\}\,,\\
&40\Delta_{4}(q)^3\Delta_{4}(q^5)=5\theta(q)^3\theta(q^5)+\psi_1(q)^{10}-36\psi_1(q)^5\psi_2(q)^5-\psi_2(q)^{10}\\
&\hspace*{30ex}-6\{\psi_1(q^4)^{10}-36\psi_1(q^4)^{5}\psi_2(q^4)^{5}-\psi_2(q^4)^{10}\}\,,\\
&48\psi_1(q^4)^{10}=2\{5\psi_1(q)^{10}+36\psi_1(q)^5\psi_2(q)^5+\psi_2(q)^{10}\}
+45\theta^2(q^5)\{\theta(q^5)^2+\theta(q)^2\}\\
&\hspace*{20ex}-2\theta(q)\theta(q^5)\{135\theta(q^5)^2+71\theta(q)^2\}\\
&\hspace*{25ex}
-160\Delta_{4}(q)^3\Delta_{4}(q^5)+360\psi_1(q^4)^5\psi_1(q)^{5}-504\psi_2(q^4)^5\psi_2(q)^5\,,\\
&48\psi_2(q^4)^{10}=2\{\psi_1(q)^{10}-36\psi_1(q)^5\psi_2(q)^5+5\psi_2(q)^{10}\}
+45\theta^2(q^5)\{\theta(q^5)^2+\theta(q)^2\}\\
&\hspace*{20ex}+2\theta(q)\theta(q^5)\{135\theta(q^5)^2+71\theta(q)^2\}\\
&\hspace*{25ex}
+160\Delta_{4}(q)^3\Delta_{4}(q^5)-504\psi_1(q^4)^5\psi_1(q)^{5}+360\psi_2(q^4)^5\psi_2(q)^5\,.
\end{align*}

By using the above {\bf (a)}--{\bf(g)}, we can prove that the functions listed in Appendices~\ref{sec:basisrational}--\ref{sec:quasimodular} are solutions of~$(\flat_s)$.

\section{Modular linear differential equations with a~solution of CFT type}
\label{sec:basisrational}

In this appendix we list a~fundamental system of solutions of~\eqref{eqn:fourth} in a~case by case basis
when~\eqref{eqn:fourth} has a~solution of CFT type and does not have a~quasimodular 
solution of positive depth.
Among $s$ in~\eqref{eqn:cand5}, the numbers in~\eqref{eqn:cand6} have the required conditions.
The solutions of the remaining cases are given in Appendix~\ref{sec:quasimodular}.
For $r\in\mathbb{Q}$, the solution $f_r$ in this section has the form
$f_r(\tau)=q^{r-\alpha/24}(1+O(q))$, where $\alpha$ is a~formal central charge
of~\eqref{eqn:fourth}.

\medskip\noindent
{\bf (a) $\mathbf{s=-48/5}$.}
A~fundamental system of solutions of~$(\flat_{-48/5})$
 is given by 

\begin{align*}
f_{0}&\=\frac{\psi_2(q)\Delta_2(q)\big(11\psi_1(q)^{10}-66\psi_1(q)^5\psi_2(q)^5-\psi_2(q)^{10}+H_2(q)\big)}{12\eta(q)^{42/5}}\\
&\=q^{7/20}(1+14 q+119 q^2+770 q^3+4088 q^4+18676 q^5+\cdots)\,,
\\
f_{4/5}&\=\frac{\psi_1(q)\Delta_2(q)\big(-\psi_1(q)^{10}+66\psi_1(q)^5\psi_2(q)^5+11\psi_2(q)^{10}+H_2(q)\big)}{84\eta(q)^{42/5}}\\
&\=q^{23/20}\left(1+14 q+\frac{769}{7}q^2+642 q^3+3103 q^4+13078 q^5+49616 q^6+\cdots\right)\,,
\\
f_{-1/2}&\=\frac{\psi_2(q)\big\{-H_2(q)^2+192\Delta_2(q)^2+H_2(q)\big(22\psi_1(q)^{10}-132\psi_1(q)^5\psi_2(q)^5-2\psi_2(q)^{10}\big)\big\}}{21\eta(q)^{42/5}}\\
&\=q^{-3/20}\left(1+40 q+381 q^2+2865 q^3+\frac{115789}{7} q^4+81261 q^5+348612 q^6+\cdots\right)\,,
\\
f_{-7/10}&\=\frac{\psi_1(q)\big\{H_2(q)^2-192\Delta_2(q)^2+H_2(q)\big(2\psi_1(q)^{10}-132\psi_1(q)^5\psi_2(q)^5-22\psi_2(q)^{10}\big)\big\}}{3\eta(q)^{42/5}}\\
&\=q^{-7/20}\left(1-63 q-1883 q^2-18403 q^3-122388 q^4-645036 q^5-2896215 q^6+\cdots\right)\,.
\end{align*}

\medskip\noindent
{\bf (b) $\mathbf{s=-38/5}$.}
A~fundamental system of solutions of~$(\flat_{-38/5})$
 is given by 
\begin{align*}
f_{-8/15}&\=\frac{\psi_1(q)G_1(I_{15}(q),\Delta_{15}(q),I_3(q),I_3(q^5))}{16\eta(q)^{32/5}}\\
&\=q^{-4/15}\left(1-56 q-776 q^2-5088 q^3-24932 q^4-\cdots\right)\,,
\\
f_{-1/3}&\=\frac{\psi_2(q)G_1(I_{15}(q),\Delta_{15}(q),-I_3(q),-I_3(q^5))}{128\eta(q)^{32/5}}\\
&\=q^{-1/15}\Big(1+15 q+100 q^2+\frac{4629}{8}q^3+2635 q^4+\cdots\Big)\,,
\\
f_{4/5}&\=\frac{\psi_1(q)G_2(I_{15}(q),\Delta_{15}(q),I_3(q),I_3(q^5),\psi_2(q)^5)}{936493073280\Delta_3(q)^2\eta(q)^{32/5}}\\
&\=q^{16/15}\Big(1+\frac{28}{3}q+\frac{164}{3}q^2+\frac{752}{3}q^3+\frac{1955}{2}q^4+\cdots\Big)\,,
\\
f_{0}&\=\frac{\psi_2(q)G_3(I_{15}(q),\Delta_{15}(q),I_3(q),I_3(q^5),\psi_2(q)^5)}{31216435776\Delta_3(q)^2\eta(q)^{32/5}}\\
&\=q^{4/15}\left(1+8 q+56 q^2+288 q^3+1254 q^4+\cdots\right)\,,
\end{align*}
where $G_1(u,x,y)$, $G_2(u,v,x,y,w)$ and $G_3(u,v,x,y,w)$ are homogeneous
polynomials of degree $3,5$ and $5$, respectively (they are given in Appendix~\ref{sec:poly}).

\medskip\noindent
{\bf (c) $\mathbf{s=-6/5}$.}
The following functions are solutions of~$(\flat_{-6/5})$:
\begin{align*}
f_{0}&\=1\,,
\\
f_{1/5}&\=\frac{1}{5}\int^q_0 \psi_1(q)^4\psi_2(q)\big(\psi_1(q)^5-3\psi_2(q)^5\big)\frac{dq}{q}
\=q^{1/5}\left(1+\frac{1}{3}q+\frac{12}{11}q^2+\frac{11}{16}q^3+\frac{4}{7}q^4+\cdots\right)\,,
\\
f_{4/5}&\=\frac{1}{15}\int^q_0 \psi_1(q)\psi_2(q)^4\big(12\psi_1(q)^5+4\psi_2(q)^5\big)\frac{dq}{q}
\=q^{4/5}\left(1+\frac{28}{27}q+\frac{4}{7}q^2+\frac{80}{57}q^3+\frac{5}{9}q^4+\cdots\right)\,.
\end{align*}
It seems unlikely that $f_{1/5}$ and~$f_{4/5}$ are modular forms
(since denominators of the Fourier coefficients of them looks like unbounded).
Since~$(\flat_{-6/5})$ has the index~0 as a~double root,~$(\flat_{-6/5})$ has a~solution with logarithmic terms.

However, modular functions
\begin{equation*}
\frac{f_{1/5}'}{\eta^4}\=\frac{\psi_1(q)^4\psi_2(q)\big(\psi_1(q)^5-3\psi_2(q)^5\big)}{5\eta(q)^4}\,,\quad
\frac{f_{4/5}'}{\eta^4}\=\frac{ \psi_1(q)\psi_2(q)^4\big(12\psi_1(q)^5+4\psi_2(q)^5\big)}{60\eta(q)^4}
\end{equation*}
 on $\Gamma(30)$  satisfy a~third order MLDE
\begin{equation*}
f'''-\frac{1}{2}E_2f''+\Big(\frac{1}{2}E_2'-\frac{9}{100}E_4\Big)f'+\frac{19}{5400}E_6f\=0\,.
\end{equation*}
The other solution of this MLDE is given by
\begin{equation*}
\frac{\psi_1(q)^{10}-36\psi_1(q)^5\psi_2(q)^5-\psi_2(q)^{10}}{\eta(q)^4}\=q^{-1/6}\big(1 - 26 q - 126 q^2 - 500 q^3-\cdots)\,.
\end{equation*}

\medskip\noindent
{\bf (d) $\mathbf{s=-3/5}$.}
A~fundamental system of solutions of~$(\flat_{-3/5})$
is given by 
\begin{align*}
f_{0}&=\frac{\theta(q)+\theta(q^5)}{2\eta(q)^{3/5}\psi_1(q)}
	=q^{-1/40}\big(1+q+q^2+2q^3+3q^4+\cdots\big)\,,
\\
f_{4/5}&=\frac{\theta(q)-\theta(q^5)}{2\eta(q)^{3/5}\psi_2(q)}
	=q^{31/40}\big(1+q+q^2+2q^3+2q^4+\cdots\big)\,,
\\
f_{1/4}&=\frac{\Delta_4(q)+\Delta_4(q^5)}{\eta(q)^{3/5}\psi_1(q)}
	=q^{9/40}\big(1+q+2q^2+2q^3+3q^4+\cdots\big)\,, 
\\
f_{1/20}&=\frac{\Delta_4(q)-\Delta_4(q^5)}{\eta(q)^{3/5}\psi_2(q)}
	=q^{1/40}\big(1+q^2+q^3+2q^4+2q^5+\cdots\big)\,.
\end{align*}

\medskip\noindent
{\bf (e) $\mathbf{s=2/5}$.}
A~fundamental system of solutions of~$(\flat_{2/5})$
is given by 
\begin{align*}
f_{0}&\=\frac{I_{15}(q)-\Delta_{15}(q)+I_3(q)}{2\eta(q)^{8/5}\psi_1(q)}
\=q^{-1/15}(1+4 q+8 q^2+20 q^3+37 q^4+\cdots)\,,
\\
f_{4/5}&\=\frac{-I_{15}(q)+\Delta_{15}(q)+I_3(q)}{6\eta(q)^{8/5}\psi_2(q)}
\=q^{11/15}\Big(1+\frac{4}{3}q+\frac{10}{3}q^2+\frac{20}{3}q^3+\frac{38}{3}q^4+\cdots\Big)\,,
\\
f_{1/3}&\=\frac{G(I_{15}(q),\Delta_{15}(q),I_3(q),I_3(q^5))}{864\eta(q)^{8/5}\psi_1(q)}
\=q^{4/15}\Big(1+\frac{5}{2}q+6 q^2+\frac{23}{2} q^3+23 q^4+\cdots\Big)\,,
\\
f_{2/15}&\=\frac{G(-I_{15}(q),-\Delta_{15}(q),I_3(q),I_3(q^5))}{432\eta(q)^{8/5}\psi_2(q)}
\=q^{1/15}(1+2 q+7 q^2+12 q^3+26 q^4+\cdots)\,,
\end{align*}
where 
\begin{equation*}
G(u,v,x,y)=60 (3 u^3 + v^3) + 81 (u - v) (u + v) x-(5 u + 9 v - 12 x) x^2 - 
 108 (u^2 + v^2) y - 25 (7 u + 3 v)y^2+15y^3.
\end{equation*}

\medskip\noindent
{\bf (f) $\mathbf{s=6/5}$.}
A~fundamental system of solutions of~$(\flat_{6/5})$
is given by 
\begin{align*}
f_{0}	&\=\frac{\psi_1(q)\big(\psi_1(q)^5+2\psi_2(q)^5\big)}{\eta(q)^{12/5}}
	\=q^{-1/10}(1 + 8 q + 23 q^2 + 68 q^3+\cdots)\,,
\\
f_{1/5}&\=\frac{\psi_2(q)\big(2\psi_1(q)^5-\psi_2(q)^5\big)}{2\eta(q)^{12/5}}
	\=q^{1/10}\left(1+\frac{9}{2}q+16q^2+38 q^3+\cdots\right)\,,
\\
f_{2/5}&\=\frac{\psi_1(q)^4\psi_2(q)^2}{\eta(q)^{12/5}}
	\=q^{3/10}\left(1 + 4 q + 12 q^2 + 30 q^3+\cdots\right)\,,
\\
f_{4/5}&\=\frac{\psi_1(q)^2\psi_2(q)^4}{\eta(q)^{12/5}}
	\=q^{7/10}\left(1 + 2 q + 7 q^2 + 16 q^3+\cdots\right).
\end{align*}

\medskip\noindent
{\bf (g) $\mathbf{s=12/5}$.}
A~fundamental system of solutions of~$(\flat_{12/5})$
 is given by 
\begin{align*}
f_{0}	&=\frac{1}{40\eta(q)^{18/5}\psi_1(q)}\{8 \psi_2(q)^5\big(18 \psi_1(q)^5 +\psi_2(q)^5\big)
+16 \big(\psi_1(q^2)^{10} + 21\psi_1(q^2)^5\psi_2(q^2)^5- 2\psi_2(q^2)^{10}\big)\\
&
\qquad\qquad\qquad\qquad\qquad+5H_2(q) + 19H_2(q^5)\}\\
	&\=q^{-3/20}(1+18 q+81 q^2+306 q^3+909 q^4+\cdots)\,,
\\
f_{4/5}&\=\frac{1}{360\eta(q)^{18/5}\psi_2(q)}\{-8 \psi_2(q)^5\big(18 \psi_1(q)^5 +\psi_2(q)^5\big)-16 \big(\psi_1(q^2)^{10} + 21\psi_1(q^2)^5\psi_2(q^2)^5\\
&\qquad\qquad\qquad\qquad\qquad- 2\psi_2(q^2)^{10}\big)+21H_2(q)-5H_2(q^5)\}\\
&\=q^{13/20}\Big(1+\frac{34}{9}q+17 q^2+50 q^3+\frac{428}{3}q^4+\cdots\Big)\,,
\\
f_{1/2}&\=\frac{1}{6\eta(q)^{18/5}\psi_1(q)^6}\{\Delta_2(q)\big(5\psi_1(q)^5+\psi_2(q)^5
+\psi_1(q^2)^5-\psi_2(q^2)^5\big)+\Delta_2(q^5)\big(7\psi_1(q)^5-\psi_2(q)^5\\
&\qquad\qquad\qquad\qquad\qquad-\psi_1(q^2)^5-7\psi_2(q^2)^5\big)\}\\
&\=q^{7/20}\Big(1+\frac{20}{3}q+27 q^2+89 q^3+\frac{766}{3}q^4+\cdots\Big),
\\
f_{3/10}&\=\frac{1}{2\eta(q)^{18/5}\psi_2(q)^6}\{\Delta_2(q)\big(-\psi_1(q)^5
+5\psi_2(q)^5+\psi_1(q^2)^5+\psi_2(q^2)^5\big)+\Delta_2(q^5)\big(\psi_1(q)^5\\
&\qquad\qquad\qquad\qquad\qquad+7\psi_2(q)^5+7\psi_1(q^2)^5-\psi_2(q^2)^5\big)\}\\
&\=q^{3/20}(1+9 q+39 q^2+131 q^3+387 q^4+\cdots).
\end{align*}

\medskip\noindent
{\bf (h) $\mathbf{s=18/5}$.}
A~fundamental system of solutions of~$(\flat_{18/5})$
is given by 
\begin{align*}
f_{0}	&\=\frac{\psi_1(q)^2\big(\psi_1(q)^{10}+24\psi_1(q)^{5}\psi_2(q)^5-6\psi_2(q)^{10}\big)}{\eta(q)^{24/5}} 
\=q^{-1/5}(1+36 q+240 q^2+1144 q^3+\cdots)\,,
\\
f_{2/5}&\=\frac{\psi_2(q)^2\big(6\psi_1(q)^{10}+24\psi_1(q)^{5}\psi_2(q)^5-\psi_2(q)^{10}\big)}{6\eta(q)^{24/5}}
\=q^{1/5}\left(1+14 q+\frac{461}{6}q^2+330 q^3+\cdots\right)\,,
\\
f_{3/5}&\=\frac{\psi_1(q)^4\psi_2(q)^3\big(4\psi_1(q)^{5}+3\psi_2(q)^{5}\big)}{4\eta(q)^{24/5}}
\=q^{2/5}\left(1+\frac{39}{4}q+51q^2+\frac{417}{2}q^3+\cdots\right)\,,
\\
f_{4/5}&\=\frac{\psi_1(q)^3\psi_2(q)^4\big(3\psi_1(q)^{5}-4\psi_2(q)^{5}\big)}{3\eta(q)^{24/5}}
\=q^{3/5}\left(1+\frac{20}{3}q+36 q^2+136 q^3+\cdots\right)\,.
\end{align*}

\medskip\noindent
{\bf (i) $\mathbf{s=22/5}$.}
A~fundamental system of solutions of~$(\flat_{22/5})$
is given by 
\begin{align*}
f_{0}&\=\frac{G_4(I_{15}(q),\Delta_{15}(q),I_3(q),I_3(q^5))}{24\psi_1(q)\eta(q)^{28/5}}
\=q^{-7/30}(1+56 q+476 q^2+2632 q^3+11270 q^4+\cdots)\,,
\\
f_{4/5}&\=\frac{G_4(-I_{15}(q),-\Delta_{15}(q),I_3(q),I_3(q^5))}{504\psi_2(q)\eta(q)^{28/5}}\\
&\=q^{17/30}\Big(1+\frac{28}{3}q+\frac{1196}{21}q^2+\frac{752}{3}q^3+\frac{2851}{3}q^4+\cdots\Big)\,,
\\
f_{2/3}&\=\frac{G_5(I_{15}(q),\Delta_{15}(q),I_3(q),I_3(q^5),\psi_2(q)^5)}{26309472\psi_1(q)\Delta_3(q)\eta(q)^{28/5}}
\=q^{13/30}(1+12 q+73 q^2+338 q^3+\frac{9070}{7} q^4+\cdots)\,,
\\
f_{7/15}&\=\frac{G_6(I_{15}(q),\Delta_{15}(q),I_3(q),I_3(q^5),\psi_2(q)^5)}{7516992\psi_2(q)\Delta_3(q)}\\
&\=q^{7/30}\Big(1+\frac{35}{2}q+112 q^2+\frac{1099}{2}q^3+2163 q^4+\cdots\Big)\,,
\end{align*}
where $G_4(u,v,x,y)$, $G_5(u,v,x,y,w)$ and $G_6(u,v,x,y,w)$ are homogeneous
polynomials of degree $3,4$ and $4$, respectively (they are written in  Appendix~\ref{sec:poly}).

\medskip\noindent
{\bf (j) $\mathbf{s=27/5}$.}
A~fundamental system of solutions of~$(\flat_{27/5})$
is given by 
\begin{align*}
f_{0}&\=\frac{G_7(\theta(q),\theta(q^5),\psi_1(q)^5,\psi_2(q)^5)}{10\psi_1(q)\eta(q)^{33/5}}
\=q^{-11/40}(1+99 q+1122 q^2+7425 q^3+37191 q^4+\cdots)\,,
\\
f_{4/5}&\=\frac{G_8(\theta(q),\theta(q^5),\psi_1(q)^5,\psi_2(q)^5)}{330\psi_2(q)\eta(q)^{33/5}}
\=q^{21/40}\Big(1+\frac{41}{3}q+98 q^2+513 q^3+2214 q^4+\cdots\Big)\,,
\\
f_{3/4}&\=\frac{G_9(\theta(q),\theta(q^5), \psi_1(q^4)^5, \psi_2(q^4)^5, \Delta_4(q)^3\Delta_4(q^5))}{9641984\psi_1(q)\Delta_4(q)\eta(q)^{33/5}}\\
&\=q^{19/40}\Big(1+15 q+\frac{1191}{11}q^2+577 q^3+2505 q^4+\cdots\Big)\,,
\\
f_{11/20}&\=\frac{G_{10}(\theta(q),\theta(q^5), \psi_1(q^4)^5, \psi_2(q^4)^5,\Delta_4(q)^3\Delta_4(q^5))}{7888896\psi_2(q)\Delta_4(q)\eta(q)^{33/5}}\\
&\=q^{11/40}\Big(1+22 q+\frac{506}{3}q^2+957 q^3+4279 q^4+\cdots\Big)\,,
\end{align*}
where $G_7(u,v,x,y)$, $G_8(u,v,x,y)$ $G_9(u,v,x,y,w)$ and $G_{10}(u,v,x,y,w)$ are homogeneous
polynomials of degree $7,7,8$ and $8$, respectively, which are given in Appendix~\ref{sec:poly}.

\medskip\noindent
{\bf (k) $\mathbf{s=6}$.}
The following functions are solutions of~$(\flat_{6})$:
\begin{align*}
f_{0}	&\=\frac{\psi_1(q)^3\big(\psi_1(q)^{15}+126\psi_1(q)^{10}\psi_2(q)^5+117\psi_1(q)^5\psi_2(q)^{10}-12\psi_2(q)^{15}\big)}{\eta(q)^{36/5}}\\
	&\=q^{-3/10}(1+144 q+1926 q^2+14160 q^3+77499 q^4+\cdots)\,,
\\
f_{3/5}&\=\frac{\psi_2(q)^3\big(12\psi_1(q)^{15}+117\psi_1(q)^{10}\psi_2(q)^5-126\psi_1(q)^5\psi_2(q)^{10}+\psi_2(q)^{15}\big)}{12\eta(q)^{36/5}}\\
	&\=q^{3/10}\left(1+\frac{99}{4}q+210 q^2+\frac{7739}{6}q^3+6195 q^4+\cdots\right)\,,
\\
f_{4/5}&\=\frac{\psi_1(q)^4\psi_2(q)^4\big(9\psi_1(q)^{10}+26\psi_1(q)^5\psi_2(q)^5-9\psi_2(q)^{10}\big)}{9\eta(q)^{36/5}}\\
	&\=q^{1/2}\left(1+\frac{152}{9}q+134 q^2+772 q^3+\frac{10778}{3}q^4+\cdots\right)\,.
\end{align*}
The function $\eta(q)^{12}=q^{1/2}+\cdots$ is a~modular form of weight~6 on~$\Gamma(2)$. 
Then we have  
\[
\eta(q)^{12}=\eta(q)^{36/5}\psi_1(q)\psi_2(q)
\big(\psi_1(q)^{10}-11\psi_1(q)^5\psi_2(q)^5-\psi_2(q)^{10}\big)
\]
since we see that the both-sides of this equality are solutions of $\vartheta_{6}(f)=0$ by using functional equations given in Appendix~\ref{sec:modularforms} {\bf(d)}, and the leading coefficients of both $q$-series are one.
 Hence, $f_r$ is a~modular function on~$\Gamma(10)$ for each $r=0,3/5,4/5$.

We have the MLDE~$(\flat_{6})$ by applying the Serre derivation 
 $\vartheta_6(F)=F'-(E_2/2)F$ to the third order MLDE
\begin{equation*}
f'''-\frac{1}{2}E_2f''+\Big(\frac{1}{2}E_2'-\frac{9}{100}E_4\Big)f'+\frac{9}{200}E_6f\=0,
\end{equation*}
which have the formal conformal weights $0,3/5,4/5$ and formal central charge $6$.

There is a~solution of~$(\flat_{6})$ with a~logarithmic term of the form
\begin{equation*}
(2\pi \sqrt{-1}\tau)\cdot f_{4/5}-q^{3/2}\Big(\frac{2530}{81}+\frac{191600}{693}q+\frac{8906965}{4788}q^2+\frac{5783927675}{632016}q^3+\frac{385857740243}{9927918}q^4+O(q^5)\Big).
\end{equation*}

\medskip\noindent
{\bf (l) $\mathbf{s=32/5}$.}
There exist solutions of~$(\flat_{32/5})$
of the forms
\begin{align*}
f_{0}&\=\frac{\psi_1(q)^4 \left(\psi_1(q)^{15}+171\psi_1(q)^{10}\psi_2(q)^5+247\psi_1(q)^5\psi_2(q)^{10}-57\psi_2(q)^{15}\right)}{\eta(q)^{38/5}}\\
&\=q^{-19/60}(1+190 q+2831 q^2+22306 q^3+129276 q^4+611724 q^5+\cdots)\,,
\\
f_{4/5}&\=\frac{\psi_2(q)^4\left(57\psi_1(q)^{15}+247\psi_1(q)^{10}\psi_2(q)^5-171\psi_1(q)^5\psi_2(q)^{10}+\psi_2(q)^{15}\right)}{57\eta(q)^{38/5}}\\
&\=q^{29/60}\Big(1+\frac{58}{3} q+\frac{493}{3}q^2+\frac{57362}{57}q^3+\frac{14761}{3}q^4+20734 q^5+\cdots\Big)\,.
\end{align*}
The functions $f_0$ and $f_{4/5}$ satisfy~$(\sharp_{\mu(19/5)})$
and the MLDE~$(\flat_{32/5})$ is rewritten as 
\begin{equation}\label{eqn:rewritten}
\vartheta_6\circ\vartheta_4(L(f))-\frac{11}{3600}E_4L(f)\=0\,,
\end{equation}
where $L(f)$ is the left-hand side of $(\sharp_{\mu(19/5)})$.
 Hence we have other two solutions which are given by 
\begin{align*}
f_{5/6}&\=\frac{5}{144}\left(30f_{4/5}(q)\frac{\psi_1(q)^4\psi_2(q)\big(\psi_1(q)^5-3\psi_2(q)^5\big)}{\eta(q)^4}-f_{0}(q)\int_{0}^{q}f_{4/5}(q_0) \eta(q_0)^{18/5}\psi_2(q_0)\frac{dq_0}{q_0}\right)\\
&\=q^{31/60}\Big(1+\frac{200}{11}q+\frac{28647}{187}q^2+\frac{3989341}{4301}q^3+\frac{562835919}{124729}q^4+\cdots\Big)\,,
\\
f_{19/30}&\=\frac{19}{144}\left(10f_{0}(q)\frac{\psi_1(q)\psi_2(q)^4\big(3\psi_1(q)^5+\psi_2(q)^5\big)}{19\eta(q)^4}-f_{4/5}(q)\int_{0}^{q}f_0(q_0) \eta(q_0)^{18/5}\psi_1(q_0)\frac{dq_0}{q_0}\right)\\
&\=q^{19/60}\Big(1+\frac{133}{5}q+\frac{13243}{55}q^2+\frac{1454051}{935} q^3+\frac{168154408}{21505}q^4+\cdots\Big)\,.
\end{align*}

The latter two solutions are most likely not modular functions 
(since the denominators of coefficients of $q$-series look like unbounded).

\medskip\noindent
{\bf (m) $\mathbf{s=54/5}$.}
A~fundamental system of solutions of~$(\flat_{54/5})$
is given by 
\begin{align*}
f_{0}&\=\frac{P(\psi_1(q),\psi_2(q))}{\eta(q)^{12}}
\=q^{-1/2}(1+36q+2490q^2+ 38360q^3+ 398715q^4+\cdots)\,,
\\
f_{4/5}&\=\frac{Q(\psi_1(q),\psi_2(q))}{3\eta(q)^{22}}
\=q^{3/10}\left(1+\frac{212}{3}q+1312 q^2+14480 q^3+\frac{350635}{3} q^4+\cdots\right)\,,
\\
f_{1}&\=\frac{R(\psi_1(q),\psi_2(q))}{132\eta(q)^{12}}
\=q^{1/2}\left(1+\frac{95}{2}q+\frac{25360}{33}q^2+\frac{346965}{44}q^3+\frac{666770}{11}q^4+\cdots\right)\,.
\\
f_{6/5}&\=\frac{S(\psi_1(q),\psi_2(q))}{22\eta(q)^{12}}
	\=q^{7/10}\left(1+\frac{372}{11}q+\frac{10779}{22}q^2+\frac{51626}{11}q^3+\frac{379482}{11}q^4+\cdots\right)\,,
\end{align*}
where
\begin{align*}
P(x,y)&= x^{30} + 6x^{25}y^5+1875x^{20}y^{10}
-6080x^{15}y^{15}+18135x^{10}y^{20}-1038x^5y^{25}-3y^{30},\\
Q(x,y)&= x^{6}y^4 (3x^{20} +134 x^{15}y^5 +57x^{10} y^{10} + 216x^5y^{15} - 22y^{20}),\\
R(x,y)&=y^5(132x^{30}+2970x^{25}y^{5}-1520x^{20}y^{10}+7035x^{15}y^{15}
-390x^{5}y^{25}-y^{30}),\\
S(x,y)&=x^4y^6(22x^{20}+216x^{15}y^5-57 x^{10}y^{10}+134x^5y^{15}-3y^{20}).
\end{align*}

\medskip\noindent
{\bf (n) $\mathbf{s=18}$.}
A~fundamental system of solutions of~$(\flat_{18/5})$
is given by 
\begin{align*}
f_{-4/5}&\=\frac{G_{11}(\psi_1(q),\psi_2(q))}{\eta(q)^{12}}
   \=q^{-4/5}(1-216 q-90984 q^2-4550240 q^3-107053506 q^4+\cdots)\,,
\\
f_{0}&\=1\,,
\\
f_{4/5}&\=\frac{G_{12}(\psi_1(q),\psi_2(q))}{4959\eta(q)^{96/5}}
\=q^{4/5}\Big(1+\frac{248}{3}q+\frac{22360}{9}q^2+\frac{837856}{19}q^3+\frac{1680020}{3} q^4+\cdots\Big)\,.
   \\
f_{1}&\=\frac{G_{13}(\psi_1(q),\psi_2(q))}{4408\eta(q)^{96/5}}
   \=q\Big(1+63 q+\frac{31596}{19}q^2+\frac{4150739}{152}q^3+\frac{181085301}{551}q^4+\cdots)\,,
\end{align*}
where $G_{11}(x,y)$, $G_{12}(x,y)$ and $G_{13}(x,y)$ are homogeneous
polynomials of degree $23,23$ and $48$, respectively.
They are given in Appendix~\ref{sec:poly}.


\medskip\noindent
{\bf (o) $\mathbf{s=66/5}$.}
A~fundamental system of solutions of~$(\flat_{66/5})$
is given by 
\begin{align*}
f_{-1/5}&\=-\frac{P(\psi_1(q),\psi_2(q))}{4\eta(q)^{12}}
\=q^{-1/2}\Big(1 -\frac{315}{4}q-11570q^2-\frac{456545}{2}q^3-2506845 q^4-\cdots\Big)\,,
\\
f_{0}&\=\frac{Q(\psi_1(q),\psi_2(q))}{\eta(q)^{12}}
\=q^{-3/10}(1+232 q+4902 q^2+57276 q^3+490507 q^4+\cdots)\,,
\\
f_{4/5}&\=\frac{R(\psi_1(q),\psi_2(q))}{1653\eta(q)^{12}}
\=q^{1/2}\Big(1+\frac{80}{3}q+\frac{1010}{3}q^2+\frac{57840}{19}q^3+\frac{414330}{19}q^4+\cdots\Big)\,,
\\
f_{8/5}&\=\frac{S(\psi_1(q),\psi_2(q))}{551\eta(q)^{12}}
\=q^{13/10}\Big(1+24 q+\frac{5458}{19}q^2+\frac{45800}{19}q^3+\frac{8847495}{551}q^4+\cdots\Big)\,,
\end{align*}
where
\begin{align*}
P(x,y)&\=x^5 \left(4x^{25}-435x^{20}y^5-37265x^{15}y^{10}-44080 x^{10}y^{15}-2755x^5y^{20}+1653y^{25}\right),\\
Q(x,y)&\=x^9y\left(x^{20}+203x^{15}y^5-406x^{10}y^{10}-1653x^5y^{15}
+551y^{20}\right),\\
R(x,y)&\=y^5 \big(1653x^{25}+2755x^{20}y^5-44080x^{15}y^{10}+37265x^{10}y^{15}
-435x^5y^{20}-4y^{25}\big),
\\
S(x,y)&\=xy^9 \left(551x^{20}+1653x^{15}y^5-406x^{10}y^{10}-203x^5y^{15}+y^{20}\right).
\end{align*}

\medskip\noindent
{\bf (p) $\mathbf{s=-6}$.}
A~fundamental system of solutions of~$(\flat_{-6})$
is given by 
\begin{align*}
f_{-1/5}&\=\frac{\psi_1(q)^2\left(\psi_1(q)^{10}-66\psi_1(q)^5\psi_2(q)^5-11\psi_2(q)^{10}\right)}{\eta(q)^{24/5}}\\
&\=q^{-1/5}(1-54 q-395 q^2-1836 q^3-6950 q^4-\cdots)\,,
\\
f_{0}&\=\frac{\psi_1(q)^6\psi_2(q)\left(2\psi_1(q)^5+11\psi_2(q)^5\right)}{2\eta(q)^{24/5}}
\=1 + \frac{33}{2}q+100 q^2 +\frac{893}{2}q^3+1629q^4 +\cdots\,.
\\
f_{1/5}&\=\frac{\psi_2(q)^2 \left(11\psi_1(q)^{10}-66\psi_1(q)^5\psi_2(q)^5-\psi_2(q)^{10}\right)}{11\eta(q)^{24/5}}\\
&\=q^{1/5}\Big(1+4 q+\frac{296}{11}q^2+110 q^3+\frac{4344}{11}q^4+\cdots\Big)\,,
\\
f_{1}&\=\frac{\psi_1(q)\psi_2(q)^6 \left(11\psi_1(q)^5-2\psi_2(q)^5\right)}{11\eta(q)^{24/5}}
\=q\Big(1+\frac{68}{11}q+\frac{299}{11}q^2+\frac{1102}{11}q^3+\frac{3511}{11} q^4 +\cdots\Big)\,,
\end{align*}

\medskip\noindent
{\bf (q) $\mathbf{s=-8/5}$.}
The following two functions satisfy~$(\flat_{-8/5})$:
\begin{align*}
f_{0}&\=\frac{\psi_2(q)}{\eta(q)^{2/5}} \=q^{11/60}(1+q^2+q^3+q^4+q^5+\cdots)\,,
\\
f_{-1/5}&\=\frac{\psi_1(q)}{\eta(q)^{2/5}} \=q^{-1/60}(1+q+q^2+q^3+2 q^4+2 q^5+\cdots)\,.
\end{align*}
We see that $f_0$ and $f_{-1/5}$ also satisfy $(\sharp_{\mu(1/5)})$,
and $(\flat_{-8/5})$ is rewritten as
$
\vartheta_6\circ\vartheta_4(L_{1/5}(f))-
(551/3600) E_4L_{1/5}(f)\=0\,,
$
where $L_{1/5}(f)$ denotes the left-hand side of $(\sharp_{\mu(1/5)})$.
 Hence two other solutions are given by 
\begin{align*}
&f_{-1/6}
\\
&\=\frac{1}{36}\bigg\{30\frac{\psi_1(q)^7\psi_2(q)^3\big(\psi(q)_1^5-3\psi_2(q)^5\big)\big(\psi_1(q)^{10}-11\psi_1(q)^5\psi_2(q)^5-\psi_2(q)^{10}\big)^2}{\eta(q)^{14}}\\
&\quad-f_0(q)\int_{0}^{q}\frac{\psi_1(q_0)^5\big(\psi_1(q_0)^{15}+171\psi_1(q_0)^{10}\psi_2(q_0)^5+247\psi_1(q_0)^5 \psi_2(q_0)^{10}-57\psi_2(q_0)^{15}\big)}{\eta(q_0)^{4}}\frac{dq_0}{q_0}\bigg\}\\
&\=q^{1/60}\Big(1-\frac{2}{5}q+\frac{1}{11}q^2+\frac{26}{85}q^3+\frac{434}{1265}q^4+\frac{9824}{27115}q^5+\cdots\Big)\,,
\\
&f_{19/30}
\\
&\=\frac{5}{36}\bigg\{10\frac{\psi_1(q)^3\psi_2(q)^7\big(3\psi_1(q)^5+\psi_2(q)^5\big)\big(\psi_1(q)^{10}-11\psi_q(1)^5\psi_2(q)^5-\psi_2(q)^{10}\big)^2}{\eta(q)^{14}}\\
&\quad-f_{-1/5}(q)\int_{0}^{q}\frac{\psi_1(q_0)^5\big(57\psi_1(q_0)^{15}+247\psi_1(q_0)^{10}\psi_2(q_0)^5-171\psi_1(q_0)^5\psi_2(q_0)^{10}+\psi_2(q_0)^{15}\big)}{3\eta(q_0)^4}
\frac{dq_0}{q_0}\bigg\}\\
&\=q^{49/60}\Big(1+\frac{38}{33}q+\frac{371}{561}q^2+\frac{22558}{12903}q^3+\frac{383219}{374187}q^4+\frac{938830}{374187}q^5+\cdots\Big)\,.
\end{align*}
It is most likely that
the last two solutions are not modular functions 
(since the denominators of coefficients of $q$-series look like unbounded).

\section{Bases of Modular linear differential equations with solutions of quasimodular forms of positive depths}
\label{sec:quasimodular}

In this appendix we construct a~fundamental system of solutions of~\eqref{eqn:fourth}
when~\eqref{eqn:fourth} has a~quasimodular solution
of positive depth of CFT type.
Among the candidates for $s$ given in~\eqref{eqn:cand5}, the six
numbers 
$-318/5$, $-198/5$, $-138/5$, $-78/5$, $-18/5$ and $42/5$
give such cases.

\subsection{Quasimodular forms and solutions with logarithmic terms}\label{qm-and-log}
If a~MLDE $L(f)=0$ of weight $k$ has 
solutions of quasimodular forms of weight $k+r$ and depth $r\, (>0)$, 
then there are solutions of $L(f)=0$ with logarithmic terms (see~\cite{KNS}).

For a~congruence subgroup~$\Gamma\subset\mathrm{SL}_2(\mathbb{Z})$, we now suppose  that a~quasimodular form $F_k:=A_k E_2+B_k$ of weight~$k+1$ and depth~$1$ on~$\Gamma$ 
is a~solution of a~MLDE (of weight $k$ on~$\mathrm{SL}_2(\mathbb{Z})$),
where $A_k$ and $B_k$ are modular forms on~$\Gamma$  of weight~$k-1$ and $k+1$, respectively. 
Then the function $G_k=G(F_k)=(2\pi\sqrt{-1})^{-1}(F_k\log q +12 A_k)$ is a solution 
because 
\begin{equation*}
F_k\bigg|_k \begin{pmatrix}a&b\\ c&d\end{pmatrix}\,=c\,G_k+d\,F_k \quad\text{for}\ \begin{pmatrix}a&b\\ c&d\end{pmatrix}\in\Gamma
\end{equation*}
is a solution by modular invariance property.
We find that
\begin{equation*}
\left. \begin{pmatrix}F_k\\ G_k\end{pmatrix}\right|_k
\begin{pmatrix}a&b\\ c&d\end{pmatrix}=
\begin{pmatrix}d&c\\ b&a\end{pmatrix}
\begin{pmatrix}F_k\\ G_k\end{pmatrix} \ \text{for\ each}\  
\begin{pmatrix}a&b\\ c&d\end{pmatrix}\in\Gamma\,.
\end{equation*}

\subsection{Solutions of modular linear differential equations of quasimodular forms of positive depths}

In this subappendix we list solutions of~$(\flat_s)$ of quasimodular forms,
where $s$ is listed in the beginning of Appendix~\ref{sec:quasimodular}.
By using construction in Subappendix~\ref{qm-and-log}, we have a~fundamental
system of the space of solutions of~$(\flat_s)$.

\medskip\noindent
{\bf (a) $\mathbf{s=-318/5}$.}
Let $\psi_1$ and $\psi_2$ be the modular functions defined in Appendix~\ref{sec:modularforms}.
A set of two linearly independent solutions  of~$(\flat_{-318/5})$
 is given by 
\begin{align*}
f_{0}&\=\frac{F_1\big(\psi_1(q),\psi_2(q)\big)'}{50841895104\eta(q)^{192/5}}+\frac{F_2\big(\psi_1(q),\psi_2(q)\big)}{419325701671800\eta(q)^{312/5}}\\
&\= q^{13/5}(1+260 q+30056 q^2+2119676 q^3+104823121 q^4+\cdots)\,,
\\
f_{4/5}&\=\frac{F_1\big(\psi_2(q),-\psi_1(q)\big)'}{28346417433013680\eta(q)^{312/5}}+\frac{F_2\big(\psi_2(q),-\psi_1(q)\big)}{5451234121733400\eta(q)^{312/5}}\\
&\=q^{17/5}(1+236 q+25306 q^2+1680916 q^3+79143742 q^4+\cdots\Big)\,,
\end{align*}
where $F_1(x,y)$ and $F_2(x,y)$ are homogeneous
polynomials of degree $151$ and $161$, respectively.
(They are given in Appendix~\ref{sec:poly}).
The solutions $f_0$ and $f_{4/5}$ are quasimodular forms of weight 1 and depth 1.
The system $(f_0,f_{4/5},G(f_0),G(f_{4/5}))$ is a~fundamental system of solutions
of $(\flat_{-318/5})$.
Here $G(*)$ is a~function with a~logarithmic term defined in Subappendix~\ref{qm-and-log}.

\medskip\noindent
{\bf (b) $\mathbf{s=-198/5}$.}
A~set of two linearly independent solutions of~$(\flat_{-198/5})$
 is given by 
\begin{align*}
f_{0}&\=\frac{F_3\big(\psi_1(q),\psi_2(q)\big)'}{50841895104\eta(q)^{192/5}}+\frac{F_4\big(\psi_1(q),\psi_2(q)\big)}{15888092220\eta(q)^{192/5}}\\
&\= q^{8/5}(1+144 q+8880 q^2+331840 q^3+8770284 q^4+\cdots)\,,
\\
f_{4/5}&\=\frac{F_3\big(\psi_2(q),-\psi_1(q)\big)'}{4236824592\eta(q)^{192/5}}+\frac{F_4\big(\psi_2(q),-\psi_1(q)\big)}{1324007685\eta(q)^{192/5}}\\
&\=q^{12/5}\Big(1+\frac{380}{3}q+7164 q^2+251344 q^3+\frac{18958205}{3}q^4+\cdots\Big)\,,
\end{align*}
where $F_3(x,y)$ and $F_4(x,y)$ are homogeneous
polynomials of degree $91$ and $101$, respectively (they are given in Appendix~\ref{sec:poly}).
The solutions $f_0$ and $f_{4/5}$ are quasimodular forms of weight 1 and depth 1.
The system $(f_0,f_{4/5},G(f_0),G(f_{4/5}))$ is a~fundamental system of solutions
of $(\flat_{-198/5})$.

\medskip\noindent
{\bf (c) $\mathbf{s=-138/5}$.}
A~set of two linearly independent solutions of~$(\flat_{-138/5})$
 is given by 
\begin{align*}
f_{0}&\=\frac{P\big(\psi_1(q),\psi_2(q)\big)'}{48360312\eta(q)^{132/5}}+\frac{Q\big(\psi_1(q),\psi_2(q)\big)}{21981960\eta(q)^{132/5}}\\
&\=q^{11/10}(1+88 q+3256 q^2+74360 q^3+1232814 q^4+\cdots)\,,
\\
f_{4/5}&\=\frac{P\big(\psi_2(q),-\psi_1(q)\big)'}{4396392\eta(q)^{132/5}}+\frac{Q\big(\psi_2(q),-\psi_1(q)\big)}{1998360\eta(q)^{132/5}}\\
&\=q^{19/10}(1+76 q+2584 q^2+55568 q^3+876329 q^4+\cdots)\,,
\end{align*}
where 
{\small
\begin{align*}
P(x,y)&\=x(2 x^{60}-423 x^{55} y^5+62386 x^{50} y^{10}+17130550 x^{45} y^{15}+257495595 x^{40} y^{20}\\
&\qquad+449723758 x^{35} y^{25}+110860458 x^{30} y^{30}+274470804 x^{25} y^{35}-44451050 x^{20} y^{40}\\
&\qquad-31757435 x^{15} y^{45}+1691208 x^{10} y^{50}-3542 x^5 y^{55}-11 y^{60})\,,\\
Q(x,y)&\=x y^5(159 x^{65}-51748 x^{60} y^5-1801195 x^{55} y^{10}-11218651 x^{50} y^{15}+431478190 x^{45} y^{20}\\
&\qquad-348592259 x^{40} y^{25}+1688812058 x^{35} y^{30}+402533490 x^{30} y^{35}+852888611 x^{25} y^{40}\\
&\qquad-3573970 x^{20} y^{45}-4304879 x^{15} y^{50}+269233 x^{10} y^{55}+1760 x^5 y^{60}+y^{65})\,.
\end{align*}
}
The solutions $f_0$ and $f_{4/5}$ are quasimodular forms of weight 1 and depth 1.
Moreover, the system $(f_0, f_{4/5}, G(f_0), G(f_{4/5}))$ is a~fundamental system of solutions
of $(\flat_{-138/5})$.

\medskip\noindent
{\bf (d) $\mathbf{s=-78/5}$.}
A~set of linearly independent solutions  of~$(\flat_{-78/5})$
 is given by 
\begin{align*}
f_{0}&\=\frac{P(\psi_1(q),\psi_2(q))'}{2604\eta(q)^{72/5}}
-\frac{Q(\psi_1(q),\psi_2(q))}{2170\eta(q)^{72/5}}
\=q^{3/5}(1+36 q+576 q^2+6312 q^3+53739 q^4+\cdots)\,,
\\
f_{4/5}&\=-\frac{P(\psi_2(q),-\psi_1(q))}{23436\eta(q)^{72/5}}
+\frac{Q(\psi_2(q),-\psi_1(q))}{19530\eta(q)^{72/5}}\\
&\=q^{7/5}\Big(1+\frac{284}{9}q+476 q^2+4888 q^3+\frac{117116}{3}q^4 +\cdots\Big)\,,
\end{align*}
where
\begin{align*}
P(x,y)&\=y (6x^{30}+1642x^{25}y^5-1440x^{20}y^{10}
-20615x^{15}y^{15}+12665x^{10}y^{20}-123x^5y^{25}-y^{30}),\\
Q(x,y)&\=x^5 y(x^{10}-11x^5y^5-y^{10}) (x^{25}-435x^{20}y^5-6670x^{15}y^{10}
-3335x^5y^{20}+87y^{25}).
\end{align*}
The solutions $f_0$ and $f_{4/5}$ are quasimodular forms of weight 1 and depth 1.
Moreover, the system $(f_0,f_{4/5},G(f_0),G(f_{4/5}))$ is a~fundamental system of solutions
of $(\flat_{-78/5})$.

\medskip\noindent
{\bf (e) $\mathbf{s=-18/5}$.}
A~set of linearly independent solutions  of~$(\flat_{-18/5})$
 is given by 
\begin{align*}
f_{0}&\=\frac{5\psi_2(q)'}{\eta(q)^{12/5}}=q^{1/10}(1+6 q^2+16 q^3+36 q^4+72 q^5+\cdots)\,,
\\
f_{4/5}&\=\frac{5\psi_1(q)'}{3\eta(q)^{12/5}}=q^{9/10}\Big(1+\frac{8}{3}q+6 q^2+16 q^3+\frac{101}{3} q^4+72 q^5+\cdots\Big)\,.
\end{align*}
The solutions $f_0$ and $f_{4/5}$ are quasimodular forms of weight 1 and depth 1.
Moreover, the system $(f_0,f_{4/5},G(f_0),G(f_{4/5}))$ is a~fundamental system of solutions
of $(\flat_{-18/5})$.

\medskip\noindent
{\bf (f) $\mathbf{s=42/5}$.}
A~set of linearly independent solutions  of~$(\flat_{42/5})$
is given by 
\begin{align*}
f_0&\=\frac{5\left\{\psi_2(q)^4 \left(57\psi_1(q)^{15}+247\psi_1(q)^{10}\psi_2(q)^5-171 \psi_1(q)^5\psi_2(q)^{10}+\psi_2(q)^{15}\right)\right\}'}{28\eta(q)^{48/5}}\\
&\=q^{2/5}(1+36 q+436 q^2+3536 q^3+21912 q^4+113760 q^5+\cdots)\,,
\\
f_{1/5}&\=\frac{5\left\{\psi_1(q)^4 \left(\psi_1(q)^{15}+171\psi_1(q)^{10}\psi_2(q)^5+247 \psi_1(q)^5\psi_2(q)^{10}-57\psi_2(q)^{15}\right)\right\}'}{912\eta(q)^{48/5}}\\
&\= q^{3/5}\Big(1+25 q+276 q^2+\frac{8379}{4}q^3+12481 q^4+62859 q^5+\cdots\Big)\,.
\end{align*}
The solutions $f_0$ and $f_{1/5}$ are quasimodular forms of weight 1 and depth 1.
Moreover, the system $(f_0,f_{4/5},G(f_0),G(f_{4/5}))$ is a~fundamental system of solutions
of $(\flat_{-42/5})$.

\begin{remark}
The numbers $s=-318/5,-198/5,-138/5,-78/5,-48/5$ and $-38/5$ in~\eqref{eqn:cand5} has the form
$s=-2(5h^\vee+9)/5$ with the dual Coxeter numbers $h^\vee$ of 
the Lie algebras $E_8,E_7,E_6,D_4,A_2$ and $A_1$
(the simply-laced Lie algebras in the Deligne exceptional series).
When we substitute $h^\vee=3/2,4,9$ and $24$ (the dual Coxeter numbers of
the non simply-laced Lie algebras $G_2$ and $F_4$ in the Deligne exceptional series
and formal dual Coxeter numbers of $A_{1/2}$ and $E_{7+1/2}$ \cite{LM,Kaw1}) in the above expression of $s$, we have
$s=-33/5,-58/5,-108/5$ and $-258/5$.
Let $s$ be one of the numbers $-33/5,-58/5,-108/5$ and $-258/5$.
We believe that the modular linear differential equation~\eqref{eqn:fourth}
has a~solution of the form $q^{\alpha_1}(5+\sum_{n=1}^\infty a_nq^n)$
with $a_n\in\mathbb{Z}_{\geq 0}$.
We checked it for small values of $n$: we see that $a_n\in\mathbb{Z}_{\geq 0}$ for $1\leq n<100$.

The MLDEs $(\flat_{282/5})$ and $(\flat_{132/5})$ can also have solutions of the form $q^{\alpha_1}(5+\sum_{n=1}^\infty a_nq^n)$
with $a_n\in\mathbb{Z}_{\geq 0}$.
\end{remark}

\section{The polynomials which expresses the solutions of modular linear differential equations}\label{sec:poly}

In this appendix we define the polynomials appearing in the solutions 
of~\eqref{eqn:fourth} in Appendices~\ref{sec:basisrational}--\ref{sec:quasimodular}.

\begin{align*}
G_1(u,v,x,y)&\=180 (7 u^3 + v^3)-9 (11 u^2 - 31 v^2) x\\
&\qquad-x^2 (213 u + 39 v + 34 x)+162 (u^2 + v^2) y-5 y^2 (195 u - 15 v + 17 y)\,,
\\
 G_2(u,v,x,y,w)&\=-720 (223958344 u^5 + 406176133 v^5)\\
 &\qquad+(3184319633 u^4 - 197068307233 v^4 + 33362124000 w^4) x\\
 &\qquad+60 (1353945577 u^3 + 878294370 v^3) x^2+6 (6929323501 v^2 - 1198443520 w^2) x^3\\
 &\qquad-x^4 (1858938780 v + 2043544673 x)+15 (2134148833 u^4 - 1674476673 v^4) y\\
 &\qquad-1500 (154243777 u^3 + 15399378 v^3) y^2-10 (11725793962 u^2 + 3577812017 v^2) y^3\\
 &\qquad+7500 (43261956 u + 12682717 v) y^4+71019200725 y^5\,,
\\
G_3(u,v,x,y,w)&\=144 (699752744 u^5 - 2362467 v^5)\\
&\qquad+(361371889 u^4 - 1097126369 v^4 + 1211988000 w^4) x\\
&\qquad+12 (498413977 u^3 + 838363630 v^3) x^2+6 (574859825 v^2 + 326289904 w^2) x^3\\
&\qquad+(1446719244 v - 436791181 x) x^4+15 (158488301 u^4 + 365080659 v^4) y\\
&\qquad-300 (475629839 u^3 + 83270830 v^3) y^2-10 (940591436 u^2 + 521391871 v^2) y^3\\
&\qquad+1500 (32993592 u + 10616473 v) y^4+5519145425 y^5\,,\\
%
%
G_4(u,v,x,y)&\=180 u^3 + 27 u^2 (3 x - 4 y) + 7 u (x^2 - 25 y^2)\\
&\qquad + 
 3 (20 v^3 + 8 x^3 + 5 y^3 - 9 v^2 (3 x + 4 y) - v (7 x^2 + 25 y^2))\,,
\\
G_5(u,v,x,y,w)&\=927276 u^4 + 27002724 v^4 - 1508220 u^2 y^2 + 9998940 v^2 y^2\\
&\qquad-8 u^3 (719719 x + 1768043 y) - 8 v^3 (1664689 x + 9604649 y)\\
&\qquad +4 u (1080301 x^3 + 3999985 y^3)+ 4 v (390017 x^3 + 7292255 y^3)\\
&\qquad -49 (120000 w^4 - 137104 w^2 x^2 + 7621 x^4 - 10925 y^4)\,,
\\
G_6(u,v,x,y,w)&\=927276 u^4 + 27002724 v^4 - 1508220 u^2 y^2 + 9998940 v^2 y^2\\
&\qquad-8 u^3 (257897 x + 1184773 y) - 8 v^3 (733055 x + 6576439 y)\\
&\qquad +4 u (548167 x^3 + 2441935 y^3)+ 4 v (222227 x^3 + 2618105 y^3)\\
&\qquad-49 (120000 w^4 - 137104 w^2 x^2 + 7621 x^4 - 10925 y^4)\,,\\
%
%
G_7(u,v,x,y)&\=12 u^5 v^2 + 60 u^4v^3+3 u^2 v (4 v^4 + 61 x^2 - 150 x y - 39 y^2)\\
&\qquad-60 u v^2 (5 v^4 + 5 x^2 + 5 x y + 3 y^2)-12 v^3 (5 v^4 - 4 x^2 + 44 x y + 4 y^2)\\
&\qquad+ u^3 (300 v^4 + 55 x^2 - 18 x y + 15 y^2),
\\
G_8(u,v,x,y)&\=12 u^5 v^2 - 60 u^4 v^3 + 12 v^3 (5 v^4 + 4 x^2 - 44 x y - 4 y^2)\\
&\qquad-60 u v^2 (5 v^4 + 3 x^2 - 5 x y + 5 y^2) +u^3 (300 v^4 + 15 x^2 + 18 x y + 55 y^2)\\
&\qquad-3 u^2 v (4 v^4 - 39 x^2 + 150 x y + 61 y^2),
\\
G_9(u,v,x,y,w)&\=738713 u^8+105264 u^7 v-109300 u^6 v^2-339760 u^5 v^3\\
&\qquad+2 u^4 \left(764115 v^4+1803136 w-708480 x^2\right)+80 u^3 \left(33729 v^5-349760 v w\right)\\
&\qquad-20 u^2 \left(136625 v^6-763776 v^2 w\right)-2546000 u v^7+3415625 v^8\\
&\qquad+6400 v^4 \left(695 w+1216 x^2-109 y^2\right)\\
&\qquad+64 \big\{2023440 w^2-8 w \left(8049 x^2+2936 x y+32151 y^2\right)\\
&\qquad\quad-142563 x^4-33071418 x^2 y^2-228288 x y^3-34771 y^4\big\}\,,
\\
G_{10}(u,v,x,y,w)&\=1564813 u^8-315792 u^7 v-127492 u^6 v^2+1019280 u^5 v^3\\
&\qquad-6 u^4 \left(194875 v^4+1803136 w-708480 x^2\right)-16 u^3 \left(505935 v^5+2039872 v w\right)\\
&\qquad-20 u^2 \left(159365 v^6+2291328 v^2 w\right)+7638000 u v^7+3984125 v^8\\
&\qquad-19200 v^4 \left(695 w+1216 x^2-109 y^2\right)\\
&\qquad+64 \big\{(1558352 w^2+40 w \left(29311 x^2+301704 x y-5191 y^2\right)\\
&\qquad\quad+277889 x^4-65865346 x^2 y^2+684864 x y^3-45487 y^4\big\}\,,\\
G_{11}(x,y)&\=x^8 (x^{15} + 171x^{10}y^5 + 247x^5y^{10} - 57y^{15}) (x^{25} - 435x^{20}y^{5} -  6670x^{15}y^{10} \\
&\qquad- 3335x^5y^{20} + 87y^{25}),\\
G_{12}(x,y)&\=y^8 \big(57 x^{15} + 247x^{10} y^5 - 171 x^5 y^{10} + y^{15}\big) \big(87 x^{25} + 3335 x^{20} y^5 + 
   6670 x^{10} y^{15} \\
   &\qquad- 435x^5 y^{20} -y^{25})\\
G_{13}(x,y)&\=x^4 y^9 (4408x^{35} + 105792x^{30}y^{5} + 255432x^{25}y^{10} - 39585x^{20}y^{15} + 
   224280x^{15}y^{20}\\
   &\qquad - 110124x^{10}y^{25} 
+ 4144x^{5}y^{30} - 6y^{35}).
\end{align*}

{\small
\begin{align*}
F_1(x,y)&\=y(26 x^{150}-14027 x^{145} y^5+4302519 x^{140} y^{10}-1108131323 x^{135} y^{15}+352658709221 x^{130} y^{20}\\
&\qquad+410367539820195 x^{125} y^{25}+40445085344188305 x^{120} y^{30}+1299280606676985415 x^{115} y^{35}\\
&\qquad+15619328979178395195 x^{110} y^{40}+55612071446837334235 x^{105} y^{45}\\
&\qquad-87018500083008224035 x^{100} y^{50}-533269157005860756105 x^{95}y^{55}\\
&\qquad-788818605648335112065 x^{90} y^{60}-1013858166164551300020 x^{85} y^{65}\\
&\qquad-296448112259269579710 x^{80} y^{70}-1142271481069396890795 x^{75} y^{75}\\
&\qquad+1168196293014645641265 x^{70} y^{80}-2113215283984202178330 x^{65} y^{85}\\
&\qquad+2233117451140029299360 x^{60} y^{90}-1786660513705581667220 x^{55} y^{95}\\
&\qquad+791868758851817951430 x^{50} y^{100}-151112565855362432510 x^{45} y^{105}\\
&\qquad+12466416846335713220 x^{40} y^{110}-431194652630175240 x^{35} y^{115}\\
&\qquad+5161178221901730 x^{30} y^{120}-5636816703218 x^{25} y^{125}\\
&\qquad-22864978164 x^{20} y^{130}-122736542 x^{15} y^{135}-623536 x^{10} y^{140}-2403 x^5 y^{145}-5 y^{150})\,,\\
F_2(x,y)&\=x^5 y(-x^{155}+2795 x^{150} y^5-1610905 x^{145} y^{10}+626661790 x^{140} y^{15}-272934105110 x^{135} y^{20}\\
&\qquad+5921185002218 x^{130} y^{25}+2403876207873450 x^{125} y^{30}+76784270584164950 x^{120} y^{35}\\
&\qquad+772577150281220900 x^{115} y^{40}-2601547321924922600 x^{110} y^{45}\\
&\qquad-132050328432197323010 x^{105} y^{50}-761932811447226146950 x^{100} y^{55}\\
&\qquad-646462771950208704200 x^{95} y^{60}-2253288292561296417525 x^{90} y^{65}\\
&\qquad-1120714651267256252025 x^{85} y^{70}-4859622107798545477335 x^{80} y^{75}\\
&\qquad+1377006603588253977975 x^{75} y^{80}-5767727038509251934525 x^{70} y^{85}\\
&\qquad+2918064034868992877075 x^{65} y^{90}-2471409054146940922300 x^{60} y^{95}\\
&\qquad+203085067837087672160 x^{55} y^{100}-317362540657481937650 x^{50} y^{105}\\
&\qquad+61492360597697519350 x^{45} y^{110}+350110675044144200 x^{40} y^{115}\\
&\qquad-419469534314529300 x^{35} y^{120}+21455464506696458 x^{30} y^{125}\\
&\qquad-543093587090560 x^{25} y^{130}+5212276736290 x^{20} y^{135}+16214520280 x^{15} y^{140}\\
&\qquad+62528105 x^{10} y^{145}+202921 x^5 y^{150}+375 y^{155})\,,\\
%
F_3(x,y)&\=-x (3 x^{90}-903 x^{85} y^5+159961 x^{80} y^{10}-28650647 x^{75} y^{15}-13575756015 x^{70} y^{20}\\
&\qquad-474499473257 x^{65} y^{25}-4010788885483 x^{60} y^{30}-7747481577039 x^{55} y^{35}-5318572158667 x^{50} y^{40}\\
&\qquad-5669746887385 x^{45} y^{45}+1236811165611 x^{40} y^{50}-3418602513421 x^{35} y^{55}+633550350937 x^{30} y^{60}\\
&\qquad+416362598871 x^{25} y^{65}-52095354335 x^{20} y^{70}\\
&\qquad+1196042149 x^{15} y^{75}-1716384 x^{10} y^{80}-6032 x^5 y^{85}-16 y^{90})\,,\\
F_4(x,y)&\=-x y^5(231 x^{95}-87604 x^{90} y^5+25209980 x^{85} y^{10}+1278608345 x^{80} y^{15}+28385930062 x^{75} y^{20}\\
&\qquad-156161181524 x^{70} y^{25}-3918089501554 x^{65} y^{30}+2542402356410 x^{60} y^{35}-21700694333270 x^{55} y^{40}\\
&\qquad-5108101637858 x^{50} y^{45}-21968910753568 x^{45} y^{50}-1511305888168 x^{40} y^{55}\\
&\qquad-8123129966420 x^{35} y^{60}+467241056080 x^{30} y^{65}+38837423474 x^{25} y^{70}\\
&\qquad-5177696692 x^{20} y^{75}+111522673 x^{15} y^{80}+1101610 x^{10} y^{85}+2000 x^5 y^{90}+y^{95})\,.
\end{align*}
}

\subsection*{Acknowledgment}
The authors would like to express their thanks to Prof.~K.~Nagatomo for
helpful advice and continuous encouragement.
K.K.~would like to thank Prof.~C.H.~Lam, Y.~Arike and N.~Genra for useful discussions.
K.K.'s research is partly supported by the Australian Research Council Discovery Project DP160101520.
Y.S.~was supported in part by Japan Society for the Promotion of Science, Grant-in-Aid for Scientific Research (S) 16H06336.

\end{document}